\documentclass[11pt]{article}

\usepackage{amsfonts}
\usepackage{amssymb}
\usepackage{amsmath}
\usepackage{amsthm}
\usepackage{latexsym}
\usepackage{verbatim}
\usepackage{epsfig}
\usepackage{amsbsy}
\usepackage{amscd}
\usepackage{bm}

\usepackage{graphicx}

\usepackage{ifthen} 
\usepackage{xcolor}
\usepackage[colorlinks=true, linkcolor=teal, anchorcolor=teal,
            citecolor=blue,  filecolor=blue, menucolor=blue, pagecolor=teal,											urlcolor=blue, plainpages=false, pdfpagelabels]{hyperref}

\newcommand{\mymail}[1]{\href{mailto:#1}{\texttt{#1}}}

\usepackage{fancyhdr}
\usepackage{currfile}
\fancypagestyle{firststyle}{

\cfoot{{\footnotesize \texttt{\currfilename}, version: \texttt{\today} }} 
}

\newcommand{\setauthA}[1]{\def\authA{#1}}
\newcommand{\setauthB}[1]{\def\authB{#1}}

\def\printA{\begin{tabular}{l} \authA \end{tabular}}
\def\printB{\begin{tabular}{l} \authB \end{tabular}}

\newcommand{\makemytitle}[1]{\begin{center}{\textsf{\LARGE #1}}
  \end{center}
}

\usepackage[sf,bf]{titlesec}

\usepackage{algorithmic}
\usepackage{algorithm}

\usepackage[nonamebreak,square,numbers,sort]{natbib}

\usepackage[letterpaper, textheight = 676pt]{geometry}
\providecommand{\mc}[1]{\mathcal#1}
\providecommand{\mc}[1]{\mathcal#1}

\newcommand{\R}{{\mathbb R}}

\DeclareMathOperator{\E}{\mathbf{E}}
\DeclareMathOperator{\p}{\mathbf{P}}

\DeclareMathOperator{\tr}{tr}

\providecommand{\T}{\top} 

\DeclareMathOperator*{\argmin}{argmin}
\DeclareMathOperator*{\argmax}{argmax}

\providecommand{\wt}[1]{\widetilde{#1}}
\providecommand{\wh}[1]{\widehat{#1}}

\providecommand{\norm}[1]{\left \lVert#1 \right \rVert}
\providecommand{\nnorm}[1]{ \lVert#1 \rVert}
\newcommand{\scp}[2]{\left\langle#1, #2\right\rangle}
\newcommand{\nscp}[2]{\langle#1, #2\rangle}

\newcommand{\blanco}[1]{  }

\newcommand{\deriv}[3]{%
\ifthenelse{#1 = 1}{\frac{d\,#2}{d\,#3}}{\frac{d^{{#1}} #2}{d{#3}^{{#1}}}}
}

\newcommand{\partials}[3]{%
\ifthenelse{#1 = 1}{\frac{\partial\,#2}{\partial\,#3}}{\frac{\partial^{#1}
    #2}{\partial#3^{#1}}}
} 

\def\su{\sum_{i=1}^n}

\def \coloneq{\mathrel{\mathop:}=}
\def \invcoloneq{=\mathrel{\mathop:}}
\def \eps{\varepsilon}

\newtheorem{theo}{Theorem}
\newtheorem{propo}{Theorem}
\newtheorem{definitio}{Theorem}

\newtheorem{lemmachen}{Theorem}
\newtheorem{corollary}{Theorem}

\newtheorem{defn}[definitio]{Definition}

\newtheorem{prop}[propo]{Proposition}
\newtheorem{corro}[corollary]{Corollary}
 \newtheorem{lemma}[lemmachen]{Lemma}

\def\R{\mathbb{R}}
\def\tr{\mathrm{tr}}

\DeclareMathOperator{\pr}{\texttt{P}}
\DeclareMathOperator{\pre}{\emph{\texttt{P}}}

\def\x{\mathbf{x}}
\def\y{\mathbf{y}}
\def\eps{\boldsymbol{\epsilon}}
\def\epss{\varepsilon}

\usepackage{subfigure}

\newboolean{isjournal}
\setboolean{isjournal}{true}

\newboolean{nocolors}
\setboolean{nocolors}{false}

\newboolean{usemtpro}
\setboolean{usemtpro}{false}

\ifthenelse{\boolean{nocolors}}{\hypersetup{colorlinks=false}}{}

\makeatletter
\newcommand\footnoteref[1]{\protected@xdef\@thefnmark{\ref{#1}}\@footnotemark}
\makeatother

\setauthB{{\bfseries Emanuel Ben-David} \\ Center for Statistical Research Methodology \\ U.S. Census Bureau \\ Suitland, MD 20746, USA \\ \mymail{emanuel.ben.david@census.gov}}

\setauthA{{\bfseries Martin Slawski} \\ Department of Statistics \\ George Mason University \\ Fairfax, VA 22030, USA \\ \mymail{mslawsk3@gmu.edu}}

\begin{document}
\thispagestyle{firststyle}

\makemytitle{{\bfseries {Linear Regression with Sparsely Permuted Data
    }}}
\vskip 3.5ex
{\large\begin{center}
\printA
\printB
\end{center}}

\vskip 3.5ex

\begin{abstract} In regression analysis of multivariate data, it is tacitly assumed that response and predictor variables in each observed
response-predictor pair correspond to the same entity or unit. In this paper, we consider the situation of ``permuted data'' in which this
basic correspondence has been lost. Several recent papers have considered this situation without further assumptions on the underlying permutation. In applications, the latter is often to known to have additional structure that can be leveraged. Specifically, we herein consider the common scenario of
"sparsely permuted data" in which only a small fraction of the data is affected by a mismatch between response and predictors. However, an adverse
effect already observed for sparsely permuted data is that the least squares estimator as well as other estimators not accounting for such partial mismatch are inconsistent. One approach studied in detail herein is to treat permuted data as outliers which motivates the use of robust regression formulations to estimate the regression parameter. The resulting estimate can subsequently be used to recover the permutation. A notable benefit of the proposed approach is its computational simplicity given the general lack of procedures for the above problem that are both statistically sound and computationally appealing. 
\end{abstract}

\section{Introduction}\label{sec:intro}
A largely unquestioned assumption in regression analysis with response-predictor pairs $\{ (y_i, \x_i) \}_{i = 1}^n$ is that each $y_i$ corresponds to the same statistical unit
as $\x_i$. In this paper, we consider the situation where the identifiers of the predictors (or equivalently those of the responses) are subject to an unknown permutation so that
correspondence between $y_i$ and $\x_i$ may not be be taken for granted. We refer to this situation as ``permuted data'' respectively ``sparsely permuted data'' when the permutation only affects
a small fraction of all response-predictor pairs. Restoring the original correspondence between
responses and predictors may not be achievable in practice for both computational and statistical reasons, but may also not be required for consistent estimation of regression parameters. Conversely, if the primary interest concerns recovering the permutation itself, which is the case in entity resolution, the regression model can facilitate that task.   


\begin{figure}
  \begin{minipage}{0.7\textwidth}
    \begin{tabular}{cc}
       \hskip-1ex\includegraphics[width = 0.5\textwidth]{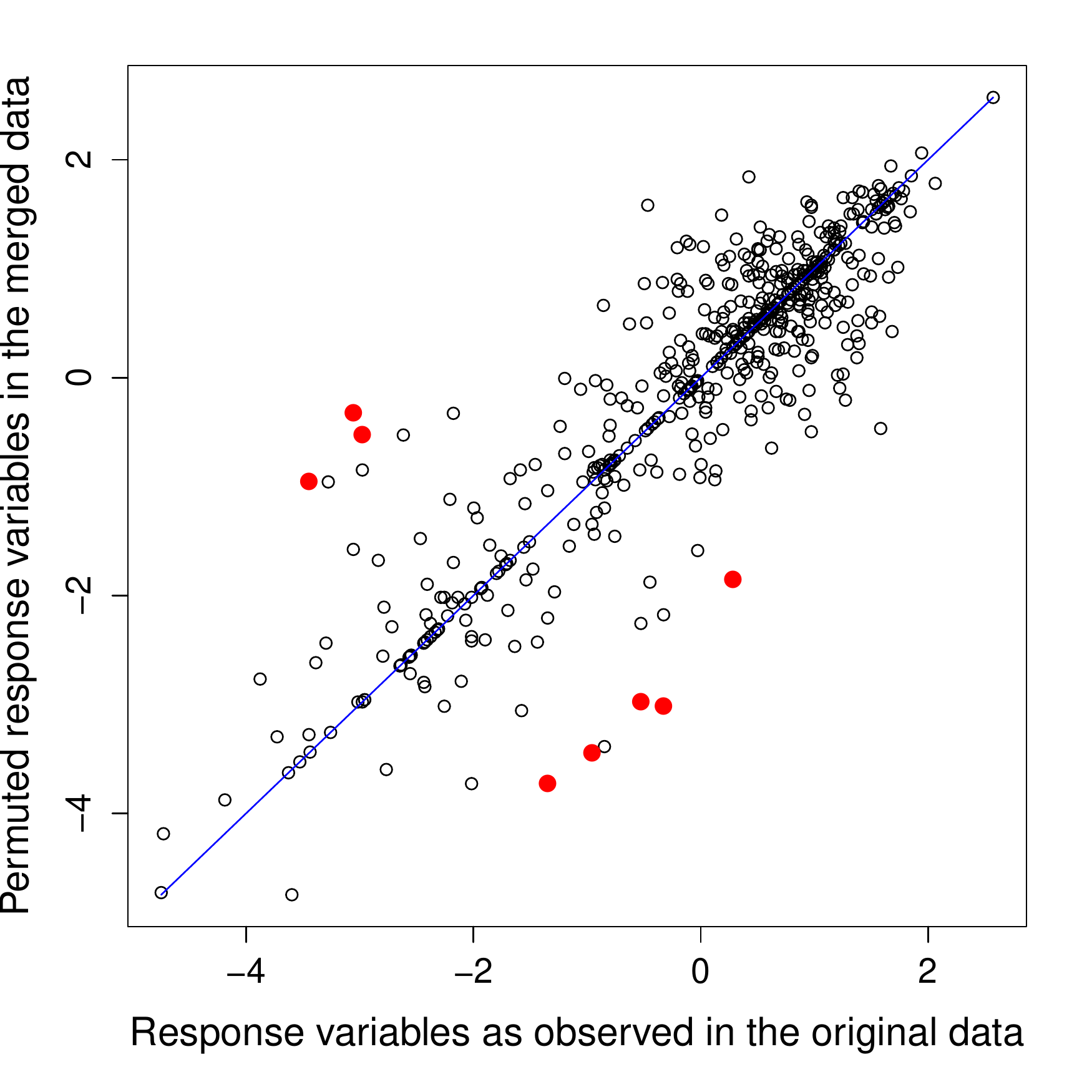} &
       \hskip-3ex\includegraphics[width = 0.5\textwidth]{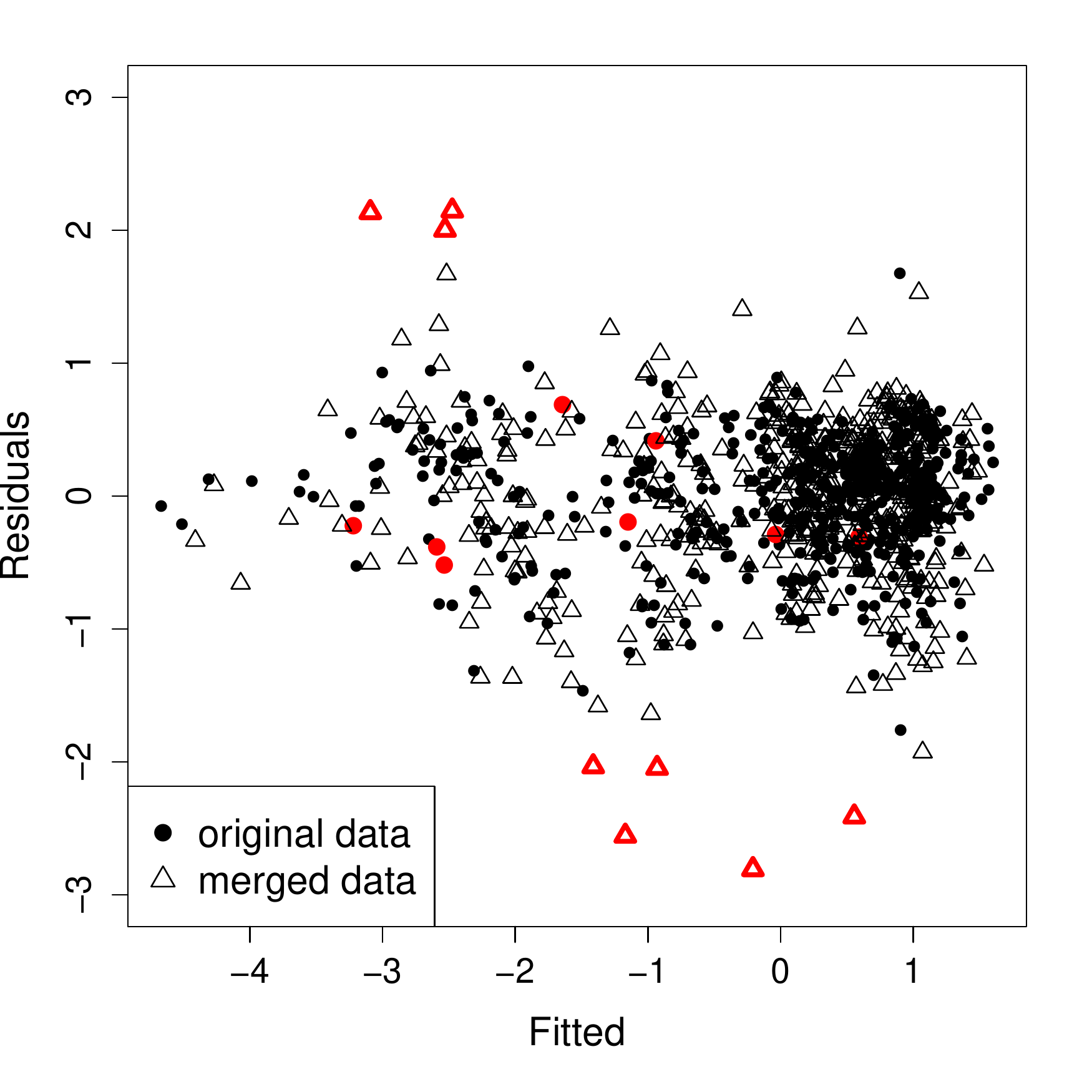}
     \end{tabular}
   \end{minipage}
   \begin{minipage}{0.28\textwidth}
$\begin{array}{cc}
   {\small \y \;\text{(original)}} & {\small \y \;\text{(merged)}} \\
   {\small \mbox{$\vdots$}}   &  {\small \mbox{$\vdots$}}  \\
   {\small \mbox{$-0.96$}}       &   {\small \mbox{$-3.45$}}                        \\
   {\small \mbox{$-3.45$}} & {\small \mbox{$-0.96$}}   \\
  {\small \mbox{$0.28$}} & {\small \mbox{$-1.86$}}   \\
 {\small \mbox{$-3.06$}} & {\small \mbox{$-0.33$}}   \\
 {\small \mbox{$-0.33$}} & {\small \mbox{$-3.02$}}   \\
 {\small \mbox{$-2.98$}} & {\small \mbox{$-0.53$}}   \\
 {\small \mbox{$-0.53$}} & {\small \mbox{$-2.98$}}   \\
   {\small \mbox{$-1.35$}} & {\small \mbox{$-3.73$}}   \\
            {\small \mbox{$\vdots$}}   &  {\small \mbox{$\vdots$}}  \\
     \end{array}$
   \end{minipage}



     
   \caption{Illustration of the effect of mismatchings in the response variable on linear regression with merged data.
     For this analysis, a subset of the El Ni\~no data set (cf.~$\S$\ref{sub:real}) was artificially broken  into two parts which were subsequently merged based on a non-unique identifier. Left: Original response vs.~response observed after data merging. Middle: Residuals vs.~fitted values in linear regression when using the original (dots) and merged data set (triangles), respectively. The circles and triangles colored in red correspond to the excerpt of the response variable depicted on the right hand side.}\label{fig:motivation_elnino}
\end{figure}

\subsection{Background and Motivation}
Large organizations that own or have access to multiple data sources regularly rely on data integration for conducting large-scale scientific projects. Available datasets are gathered or produced at different points of time and independently of one another. The main reason for combining datasets is that no single dataset contains all variables of interest that pertain the research questions. Data collected from a single source are often limited and do not have all the variables needed for statistical analysis. Limited budget, time and resources prevent each particular agency from collecting a comprehensive dataset. In most situations, however, relevant datasets from other sources are already available. For example, if a demographic survey does not have all the relevant variables for a particular research project, other existing surveys or administrative data can be used to include those missing variables.

Record linkage, or entity resolution, is an essential task in data integration. The task is to identify  which records in different datasets belong to the same entity. In this context, the term "entity" is to be understood in a broad sense, and may refer to customers, tax payers or patients, etc. Record linkage has a long history of uses in  large enterprises, government agencies and health care institutions. A business register consisting of names, addresses, and other identifying information such as total financial receipts might be constructed from tax and employment data bases; a survey of retail establishments or agricultural establishments might combine results from an area frame and a list frame. In this case, units from the area frame would need to be identified in the list frame \cite{Winkler14}.

Generally, due to the lack of unique entity identifiers in data files, record matching is based on methods of \emph{probabilistic record linkage} \cite{Fellegi69} that score similarity  between common quasi-identifiers. For example, when the data files contain health information of individuals, the quasi-identifiers can be first and last names, addresses, dates of birth.  Spelling errors, typographical variation and missing values in records inhibit exact matching and lead to matching errors: mismatches and missed-matches.  Even low matching error rates can lead to selection bias and pervasive outliers, particularly when the record linkage process includes or excludes particular types of entities and their attributes \cite{Bohensky15}. Selection bias and outliers contaminate subsequent statistical analysis. To reduce bias, it is of interest to develop statistical methods that can alleviate the adverse effects of matching errors.    

In this paper, we focus on the problem of linear regression with sparsely permuted data as it arises frequently when linking survey records to external data. The canonical example is regression estimation of a study variable $y$ on auxiliary variables $\x$ that reside in another data source such as master data or administrative records. The matching error, determined by the fraction of mismatches and missed-matches, is expected to be  only a small fraction of the sample size $n$ \cite{Winkler14}. The paper \cite{Scheuren97} considers two such situations. To model the energy economy properly, an economist might need company-specific data on the fuel and feedstocks used by companies that are only available from Agency A and corresponding microdata on the goods produced for companies that is only available from Agency B. To model the health of individuals in society, a demographer or health science policy worker might need individual-specific information on those receiving social benefits from Agencies B1, B2, and B3, corresponding income information from Agency I, and information on health services from Agencies H1 and H2. 


Regression with permuted data arises in other applications as well. The papers \cite{Pananjady2016, Abid2017} provide extensive lists including multi-target tracking in radar systems \cite{Poore06} and pose and correspondence estimation in computer vision \cite{David04}.

\subsection{Problem statement}

\begin{figure}
  \begin{center}
\begin{tabular}{cc}
$\begin{array}{c}\left[ \begin{array}{c}
(y_1, \x_1) \\
 \vdots \\
(y_n, \x_n)
                        \end{array} \right] \\
   \\
(\y, X)
 \end{array}

  \qquad \qquad$  &   $\begin{array}{c} \left[ \begin{array}{c}
(y_1, \x_{\varphi(1)}) \\
 \vdots \\
(y_n, \x_{\varphi(n)})
\end{array} \right]  \\
                     \\
(\y, \Pi^* X)
\end{array} \quad \begin{array}{c} \left[ \begin{array}{c}
(y_1, \x_{\varphi(1)}) \\
 \vdots \\
(y_n, \x_{\varphi(n)})
\end{array} \right]  \\
                     \\
(\Pi^{*\T} \y, X)
\end{array}$ \\
 observed data  $\qquad \qquad$ &  data with true correspondence
\end{tabular}
\end{center}
\vspace*{-0.2in}
\caption{Regression with lost correspondence between response and predictors.}\label{fig:permutationmodel}
\end{figure}

To setup the problem more concretely, suppose the data $(y_1,\x_1),\ldots, (y_n,\x_n)$ is obtained from matching two data files $A$ and $B$, where for each $i$, $y_i\in\R$ resides in file $A$ and $\x_i\in\R^d$ resides in file $B$. Since the process of record linkage is not error-free, some $\x_i$ may have been paired up with a non-correspondent $y_i$. If the number of such mismatches is known to be less than or equal to $k\le n$, then there is an unknown permutation $\varphi$ on $[n] \coloneq \{1,\ldots,n\}$ so that: $\varphi$ moves at most $k$ of the indices, and $(y_1,\x_{\varphi(1)}),\ldots,(y_n,\x_{\varphi(n)}) $ are independent realizations from the classical linear regression model
\begin{equation}\label{eq:lr}
y=\x^\top\beta^* +\epsilon, \: \text{where $\epsilon\sim N(0,\sigma^2)$ and $\x\perp\!\!\!\perp\epsilon$.}
\end{equation}
Let $\Pi^*$ and $\Pi^{*\T}$ be the matrix representations of $\varphi$ and its inverse, respectively, and let $X^\top=\begin{pmatrix} \x_1&\cdots&\x_d\end{pmatrix}$. Eq.~\eqref{eq:lr} readily implies that 
\begin{align}\label{eq:lrwp}
\begin{split}  
   \y &=\Pi^*X\beta^* +\boldsymbol{\epsilon},\: \text{where $\y=(y_1,\ldots,y_n)^\top$ and $\eps=(\epsilon_1,\ldots,\epsilon_n)^\top$} \\
   \Leftrightarrow \; \Pi^{*\T} \y &= X\beta^* + \Pi^{*\T} \boldsymbol{\epsilon}.
\end{split}
\end{align}
For simplicity, since the distribution of $\Pi^{*\T} \boldsymbol{\epsilon}$ and $\eps$ is the same, we do not distinguish between
$\Pi^{*\T} \boldsymbol{\epsilon}$ and $\eps$ in the sequel and write $\eps$ in place of $\Pi^{*\T} \boldsymbol{\epsilon}$.

The main question to be studied here is how the parameters $\Pi^*$ and $\beta^*$ can be accurately and efficiently estimated from the partially mismatched pairs $(y_i,\x_i)$. Efficiency here refers to computational complexity of the procedure that outputs the desired estimates. 
 Let us consider the number of mismatches which can be expressed as $d_H(\Pi^*,I_n)\coloneq|\{ i:\: \Pi_{ii}^*=0\}|$, the Hamming distance between $\Pi^*$ and the identity matrix $I_n$. We can naturally estimate $\Pi^*$ and $\beta^*$  by a least-squares approach as:
  \begin{align}\label{eq:lsq}
\begin{split}
(\widehat{\Pi},\wh{\beta})=\argmin_{\Pi,\beta} \:& \: \|\Pi X\beta-\y\|_2^2\\
\text{subject to}\: &\: d_H(\Pi,I_n)\le k,
\end{split}
\end{align}
where $\Pi$ and $\beta$ run over $\mathcal{P}_n$, the set of all permutation matrices in $\R^{n\times n}$, and $\R^d$, respectively. A permutation respectively its matrix representation $\Pi$ is said to be $k$-sparse if $d_H(\Pi, I_n)\le k$. We note that \cite{Pananjady2016} have shown that unless $d=1$ the optimization problem \eqref{eq:lsq} is NP-hard for $k=n$. The same is true for any $k$ that is defined as a fraction of $n$.
\vskip2ex
\noindent \textbf{Notation.} We here gather some notation frequently used in the present paper. For a positive integer $m$, $I_m$ denotes the
$m \times m$ identity matrix, and $\mathbb{S}^{m-1}$ denotes the unit sphere in $\R^m$. We write $|S|$ for the cardinality of a set $S$. The complement
of $S$ with respect to a given base set depending on the context is denoted by $S^c$. For a matrix $A$, $\nnorm{A}_2 = \sigma_{\max}(A)$ denotes its spectral norm respectively maximum singular value, and $\text{range}(A)$ denotes the column space of $A$. For an index set $I \subseteq [m] \coloneq \{1,\ldots,m\}$ and $v \in \R^m$, $v_{I}$ denotes the subvector of $v$ corresponding to $I$. We write 
$a \vee b = \max\{a,b\}$. Positive constants are denoted by $C$, $c$, $c_1$ etc. We make use of the usual Big-O notation in terms of $O$, $o$, $\Omega$ and $\Theta$. 

\subsection{Prior work}
Linear regression with linked data files is a common scenario in which traditional methods of statistical data analyses are error-prone. Early work in \cite{Neter65} recognizes the adverse effect of matching error on the regression analysis of linked datasets. They show that as a consequence of matching error, the ordinary least squares estimator $\wh{\beta}^{\text{ols}}$ is generally not an unbiased estimator of $\beta^*$. In the paper \cite{Neter65}, the process of matching  
is regarded as random. Letting $z_i$ denote the response that is actually in correspondence to $\x_i$, and letting $y_i$ denote
the response that has been matched to $\x_i$, the model in \cite{Neter65} assumes that for $i=1,\ldots,n$,  
 \begin{equation}\label{eq:yz}
 y_i=\begin{cases}
 	z_i&\text{with probability $q_{ii}$}\\
 	z_j&\text{with probability $q_{ij}, \; i \neq j$}.
 \end{cases}
 \end{equation}   
 Assuming that these probabilities can be estimated \cite{Neter65, Scheuren93, Scheuren97, Lahiri05, Kim12} focus on  constructing an estimator that is unbiased or less biased than $\wh{\beta}^{\text{ols}}$. In fact, in \cite{Lahiri05} the matrix $Q=(q_{ij})$ is used to construct an unbiased estimator of $\beta^*$. The papers \cite{Hof12, Gutman13, Tancredi15} propose to estimate $\beta^*$ using a Bayesian procedure. The main shortcoming of these approaches are that they rely on the assumption that the doubly stochastic matrix $Q$ is known or can be accurately estimated from the data. In addition,  while achieving reduction in bias, the proposed estimators may still be inconsistent and may have large mean squared errors. 

The classical papers \cite{DeGroot1971, DeGroot1976, DeGroot1980} and the later note \cite{Wu1998} consider the situation of permuted data under the term ``broken sample''. A broken sample is a sample of $(y_i, \x_i)$-pairs that up to some permutation of the $\{\x_i\}$ (or equivalently the $\{ y_i \}$) are generated from their joint distribution. In other words, each component of $(y_i, \x_i)$ is observed separately (as if it were generated from its marginal distribution), with possibly different orders for each component. Assuming that the $\{(y_i, \x_i)\}$ are generated from a bi-variate normal distribution up to a permutation, these papers discuss recovering the permutation or estimating the correlation parameter of the underlying bivariate normal distribution. In \cite{Bai05, Chan2001}, the authors discuss whether that parameter can be consistently estimated from a broken sample.

In recent years, computer science and engineering have witnessed a surge of interest in regression analysis of permuted data resulting from problems in multi-target tracking,  statistical seriation \cite{Flammarion16}, and unlabeled compressed sensing \cite{Unnikrishnan2015}, to mention just a few. The papers \cite{Unnikrishnan2015, Pananjady2016, Pananjady2017} are particularly important as they provide rigorous results on fundamental questions associated with the problem. We shall refer to some of these results in the subsequent sections in more detail. While the paper \cite{Collier2016} is not concerned with a regression setting, it provides a detailed analysis of the problem of finding correct matches between two sets of objects in the presence of noise, which bears some relation to the problem discussed in $\S$\ref{subsec:permutationrecovery} below.

\subsection{Summary of Contributions}
The analysis in this paper concerns estimators of $\Pi^*$ and $\beta^*$ given model \eqref{eq:lr} under the additional assumption 
$\x_i \overset{\text{i.i.d.}}{\sim} N(0, I_d)$, $i \in [n]$, as also assumed in the recent work \cite{Pananjady2016}. It is not hard to extend the results to the case where $\x_i \overset{\text{i.i.d.}}{\sim} N(0, \Sigma)$, $i \in [n]$, where $\Sigma$ is a $d \times d$ symmetric positive definite matrix.  

We first provide a negative result concerning the least squares approach \eqref{eq:lsq} if $k = n$, i.e., if no constraint is imposed on the permutation $\Pi$. It is shown that optimizing over all possible choices of $\Pi$ leads to overfitting in the sense that if $\beta^* = 0$, one still has  $\nnorm{\wh{\beta}}_2^2 = \Omega(1)$ with high probability. This result complements another one of a similar spirit in the recent paper \cite{Abid2017} who show that for $d = 1$, the least squares estimator \eqref{eq:lsq} converges almost surely to a limit different from $\beta^*$. Our result also aligns well with a minimax lower bound in \cite{Pananjady2017}. 

Altogether, these negative results provide additional justification to consider the regime of sparse permutations with $k \ll n$.
We demonstrate a bound on the estimation error $\nnorm{\wh{\beta} - \beta^*}_2$ with $\wh{\beta}$ as in \eqref{eq:lsq}. Specifically, the bound implies that the error vanishes as both $d/n$ and $k \log(n/k)/n \rightarrow 0$. In view of the fact that the optimization problem \eqref{eq:lsq} is not computationally tractable, we consider a convex relaxation that takes the form of a robust regression estimation problem as it has been considered before in different contexts \cite{Nguyen2013, She2012, FoygelMackey}, with the permuted observations here being in correspondence to gross errors. The robust regression formulation can be reduced to a specific sparse regression problem in an underdetermined setting with $n - d$ samples and $n$ parameters, one for each observation in a given sample that is affected by partial mismatches between the $\{\x_i\}_{i = 1}^n$ and the $\{ y_i\}_{i = 1}^n$. Our analysis of the robust regression problem applies generically beyond the specific setting of sparsely permuted data. Prior works \cite{Nguyen2013, FoygelMackey} have considered a more general version of the problem in which $\beta^*$ is assumed to be sparse as well, thereby being able to cover the regime $n < d$, but it is not clear whether the results in \cite{Nguyen2013, FoygelMackey} can be specialized to match those of the present paper for the traditional $n > d$ case.

While the robust regression formulation gives rise to an error bound for estimating $\beta^*$ that is comparable to that of the computationally hard formulation \eqref{eq:lsq}, it does not immediately yield an estimator for the permutation $\Pi^*$. We address this issue by adopting a two-stage approach that uses an accurate estimator $\wh{\theta}$ of $\beta^*$ to match the fitted values $\{ \x_i^{\T} \wh{\theta} \}_{i = 1}^n$ to the responses $\{ y_i \}_{i = 1}^n$. This reduces to simple sorting operations, thereby avoiding the computational challenges associated with problem \eqref{eq:lsq}. We show that our approach recovers the underlying permutation under qualitatively the same condition as in \cite{Pananjady2016} which is considerably more stringent in terms of required signal-to-noise ratio than what is required for accurate estimation of the regression coefficients $\beta^*$. We complement our result with a comparable lower bound on the signal-to-noise ratio that is required for permutation recovery even if $\beta^*$ itself is known. 

As already pointed out in \cite{Pananjady2016}, the Hamming ball constraint in \eqref{eq:lsq} does not substantially change the fundamental statistical limits of permutation recovery. However, that constraint does help in that it gives rise to a computationally efficient estimator of $\Pi^*$, whereas the statistical achievability result in \cite{Pananjady2016} is for the computationally hard estimator \eqref{eq:lsq}.




\begin{figure}
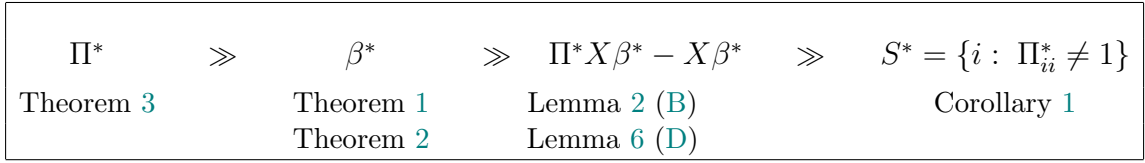

\begin{equation*}
\begin{array}{|ccccccccc|}
\hline
&&&&&&&&\\
\mbox{{\large $\Pi^*$}} & & \gg \quad &  \mbox{{\large $\beta^*$}} & & \gg \quad \mbox{{\large $\Pi^* X \beta^* - X \beta^*$}} & & \gg \quad & 
\mbox{{\large $S^* = \{i:\; \Pi_{ii}^* \neq 1\}$}} \\[1ex]
\mbox{{\normalsize Theorem \ref{theo:permutationrecovery}}} & &  & \mbox{{\normalsize Theorem \ref{theo:constrained_permutation}}}  & & \mbox{{\normalsize Lemma \ref{lem:bound_pi} (\ref{app:theo:constrained_permutation})}}  &    & & \mbox{{\normalsize Corollary \ref{corro:refitting}}}\\
& &  & 
       \mbox{{\normalsize Theorem \ref{theo:relaxation}}} & & \mbox{{\normalsize Lemma \ref{lem:bound_e} (\ref{app:theo:relaxation})}}  &   & & \\
\hline
\end{array}
\end{equation*}
\caption{Schematic overview on targets in the setting of linear regression with sparsely permuted data and pointers to corresponding results in the paper. The directions of $\gg$ indicate that if the object on the left hand side is known or can be accurately estimated, it immediately yields or simplifies estimation of the object on the right hand side.}\label{fig:diagram}
\end{figure}

\section{Main results}\label{sec:mainresults}
                  
This section is structured as follows. We start by showing that the estimator $\wh{\beta}$ in 
\eqref{eq:lsq} lacks statistical consistency if $k = n$. This sets the stage for considering the scenario $k \ll n$. We first focus on recovery of $\beta^*$, treating $\Pi^*$ as a nuisance parameter. We then present and analyze a post-processing step for recovering $\Pi^*$ given an accurate estimate of $\beta^*$.

In the sequel, unless stated otherwise, we assume Gaussian design for the $\{ \x_i \}_{i = 1}^n$:
\begin{equation*}
(\text{\textbf{G}}): \qquad \x_i \overset{\text{i.i.d.}}{\sim} N(0, I_d), \; \; i \in [n].
\end{equation*}
Our results generalize to the case $ \x_i \overset{\text{i.i.d.}}{\sim} N(0, \Sigma), \; i \in [n]$, for a symmetric positive definite matrix
$\Sigma \in \R^{d \times d}$ in a straightforward manner by defining a new regression parameter $\Sigma^{1/2} \beta^*$. An 
estimator of that parameter (as discussed in the sequel) and an estimator of $\Sigma$ can then be combined to form an estimator
of $\beta^*$. Note that estimation of $\Sigma$ is not affected by the presence of an unknown permutation.

\subsection{Least squares estimation of $(\Pi^*, \beta^*)$ without additional constraints:\\ a negative result}

Let us consider problem \eqref{eq:lsq} for $k = n$, i.e., no further constraints are imposed on the solution $(\wh{\Pi}, \wh{\beta})$. It turns out that in this case, we cannot hope for $\wh{\beta}$ to be a good estimator of $\beta^*$. As can be seen from the proof of the following proposition, complete freedom in choosing $\Pi$ to fit the data results into over-adaptation to noise even if $\beta^*$ is low-dimensional -- in fact, the phenomenon already manifests itself for $d = 1$.
\begin{prop}\label{prop:negative_result} Let $\beta^* = 0$ and $\Pi^* = I_n$, i.e., $\mathbf{y} = \eps$ and consider the estimator \eqref{eq:lsq} with $k = n$. Then there exist constants $c,C > 0$ such that with probability at least $1 - C \exp(-c n)$
\begin{equation*}
\nnorm{\wh{\beta}}_2^2 \geq \frac{n}{2n + d}\frac{\sigma^2}{32\pi^2}.
\end{equation*} 
\end{prop}   
\noindent In other words, in a pure noise setting, $\nnorm{\wh{\beta}}_2$ will be bounded away from zero by a constant (assuming $d = O(n)$). A result of a similar flavor is shown in \cite{Abid2017}. For $d = 1$, they show that the estimator $\wh{\beta}$ converges almost surely to a limit different from $\beta^*$ as $n \rightarrow \infty$, and they derive an explicit expression for that limit. Our result here is more of a qualitative nature, but it is non-asymptotic and not limited to $d = 1$. 

Both the result in \cite{Abid2017} as well as Proposition \ref{prop:negative_result} raise the question whether there exist alternative estimators that do significantly better in a regime where $\Pi^*$ can be an \emph{arbitrary} element of $\mc{P}_n$. Theorem 1 in \cite{Pananjady2017} indicates that the answer is negative: they show that for any estimator $(\wt{\Pi}, \wt{\beta})$, one has       
\begin{equation*}
\sup_{\Pi^* \in \mc{P}_n, \; \beta^* \in \R^d} \E\left[\textstyle\frac{1}{n} \nnorm{\wt{\Pi} X \wt{\beta} - \Pi^* X  \beta^*}_2^2\right] \geq c \sigma^2. 
\end{equation*}
While the above lower bound concerns estimation of $\Pi^* X  \beta^*$ rather than $\beta^*$, accurate in-sample prediction (or synonymously denoising), i.e., recovery of $\Pi^* X  \beta^*$, is typically easier than recovery of $\beta^*$. Along this direction, another negative result is shown in \cite{Hsu2017}. They demonstrate that a lower bound on the signal-to-noise ratio (SNR)
\begin{equation}\label{eq:snr}
\text{SNR} = \nnorm{\beta^*}_2^2 / \sigma^2
\end{equation}
of the order $\Omega(d / \log \log n)$ is required for any estimator of $\beta^*$ to achieve a non-trivial expected relative estimation error\footnote{Non-trivial here refers to a relative estimation error of lower order than that of the estimator $\wh{\theta} \equiv 0$.}. As shown in \cite{Pananjady2016}, the condition $\text{SNR} > n^c$ for $c$ large enough is sufficient for the solution $\wh{\Pi}$ of \eqref{eq:lsq} to recover $\Pi^*$, in which case 
$\nnorm{\wh{\beta} - \beta^*}_2^2$ scales as $O(d/n)$ as in the usual regression setting in the absence of an unknown permutation.

\subsection{Least squares estimation of $\beta^*$ if $\Pi^*$ is $k$-sparse, $k \ll n$}\label{subsec:sparsePi}
In summary, the previous section points to the fact that we cannot hope for accurate estimation of $\beta^*$ without additional
assumptions on $\Pi^*$ and/or the SNR of the problem. As motivated in the introduction, we henceforth turn our attention to 
the case of $\Pi^*$ being $k$-sparse, with $k$ ``significantly smaller'' than $n$. The allowable range of $k$ is addressed in 
our analysis presented below.       

We start by fixing notation. For $0 \leq k \leq n$, let us introduce the shorthand 
\begin{equation*}
\mc{P}_{n, k} = \{\Pi \in \mc{P}_n:\;  d_H(\Pi,I_n)\le k \}
\end{equation*}
for the constraint set of $\Pi$ in problem \eqref{eq:lsq}. Moreover, for a compact and symmetric\footnote{A set $S \subset \R^n$ is called symmetric
if $x \in S$ implies that $-x \in S$.} set $K \subset \R^n$, 
its Gaussian width is defined by 
\begin{equation}\label{eq:gaussianwidth}
w(K) = \E\nolimits_{g \sim N(0, I_n)}\left[\sup_{x \in K} |\scp{g}{x}| \right].
\end{equation}
While originating in geometric analysis, the Gaussian width is a measure of complexity that has been increasingly adopted in the analysis of high-dimensional linear inverse problems \cite{Chandrasekaran2012, Cai2016, PlanVershynin2015} in connection with Gordon's "Escape Through a Mesh Theorem" \cite{Gordon1988}, which is the key component in the proof of Theorem \ref{theo:constrained_permutation} below as well. In our setting, we use the Gaussian width \eqref{eq:gaussianwidth} in conjunction with the set 
\begin{equation}\label{eq:setT}
\mc{T} = \bigcup_{\Pi \in \mc{P}_{n,k}} \left\{\text{range}(\Pi - \Pi^{*\T}) \right\} \cap \mathbb{S}^{n-1}, 
\end{equation}
where we recall that $\Pi^{*\T} \in \mc{P}_{n,k}$ is the inverse of $\Pi^*$.
Let $\nnorm{\cdot}_0$ denote the $\ell_0$-''norm'', i.e., the number of non-zero entries of a vector. A simple observation is that for any $\Pi \in \mc{P}_{n,k}$ and any $v \in \R^n$, it holds that 
\begin{equation*}
\nnorm{\Pi v - \Pi^{*\T} v}_0 = \nnorm{(\Pi - I_n) v - (\Pi^{*\T} - I_n) v}_0 \leq \nnorm{(\Pi - I_n) v}_0 +  \nnorm{(\Pi^{*\T} - I_n) v}_0 \leq 2k.
\end{equation*} 
As a result, $\mc{T} \subseteq B_0(2k; n) \cap \mathbb{S}^{n-1}$, where for $0 \leq r \leq n$, the set $B_0(r;n) = \{v \in \R^n: \nnorm{v}_0 \leq r \}$ denotes the $\ell_0$-''ball" of radius $r$ in $\R^n$. By well-known results (cf.~\cite{PlanVershynin2013a}, Lemma 2.3), we hence have that 
\begin{equation}\label{eq:bound_width_T}
w(\mc{T}) \leq w(B_0(2k;n) \cap \mathbb{S}^{n-1}) \leq  3.5 \sqrt{2k \log(en / 2k)}. 
\end{equation}
Moreover, it is not hard to show that $w(\mc{T}) \geq w(B_0(n-k;k) \cap \mathbb{S}^{n-1})$, i.e., there is 
a lower bound on $w(\mc{T})$ of the same order. After these preparations, we are in position to state the following result. 
\begin{theo}\label{theo:constrained_permutation} Consider optimization problem \eqref{eq:lsq} and for positive integers $m$, 
denote $\nu_m = \E_{g \sim N(0, I_m)}[\nnorm{g}_2] \in \left[\frac{m}{(m+1)^{1/2}}, \; m^{1/2} \right]$ and let $\varepsilon \in (0,1/2)$ be a number such that 
\begin{equation}\label{eq:cond_samples_unrelaxed}
\nu_{n-d} - \frac{\epss}{1-\varepsilon}\nu_n \geq \frac{2}{1 - \varepsilon} w(\mc{T}). 
\end{equation}
If furthermore $n > 9 \vee 4d$, then with probability $\geq 1 - \frac{7}{2}\exp(-c_{\epss} n) - 2 \exp \left(-\frac{1}{2} (d \vee \log n) \right)$,  
\begin{equation}\label{eq:bound_samples_unrelaxed}
\nnorm{\wh{\beta} - \beta^*}_2 \leq   \frac{\sigma}{1 - \sqrt{\frac{4 d \vee \log n}{n}}} \left(\sqrt{\frac{5 (d \vee \log n)}{n}} + \frac{2 (1 + \sqrt{2}) \epss^{-2} \{w(\mc{T}) \vee \log n\}}{\sqrt{n}} \right).
\end{equation}  
\end{theo}
\noindent We start our interpretation of Theorem \ref{theo:constrained_permutation} by inspecting condition \eqref{eq:cond_samples_unrelaxed} which imposes a limit on how large $d$ and $k$ can be in relation to $n$. Roughly speaking, the requirements are $n = \Omega(d)$ and $n = \Omega(w^2(\mc{T}))$. In light of \eqref{eq:bound_width_T}, the latter condition becomes $n = \Omega(k \log(n/k))$. Specifically, let us fix $\epss = 1/4$ and $d = \alpha n$ for $\alpha \in (0,\frac{1}{4})$, then \eqref{eq:cond_samples_unrelaxed} essentially evaluates as
\begin{equation*}
n \geq \left(\frac{8}{3 (\sqrt{1 -\alpha} - 1/3)}\right)^2 w^2(\mc{T}). 
\end{equation*}
The error bound \eqref{eq:bound_samples_unrelaxed} consists of two parts. The first part equals the error one would have if the permutation $\Pi^*$ were known in advance and is thus inevitable. The second term is a bound on the excess error incurred for not knowing $\Pi^*$. That term is well controlled as long as long as the fraction of permuted observations is small relative to $n$. A crucial intermediate step in the proof of Theorem \ref{theo:constrained_permutation} is
a bound on $\nnorm{\Pi^{*\T} \y - \wh{\Pi} \y}_2$. Under an additional condition on the SNR, we may deduce from that bound that $\wh{\Pi}$ identifies the ``support''  $S^* = \{i: \, \Pi_{ii}^* \neq 1\}$ of $\Pi^*$. We may then re-fit with the corresponding observations left out, to achieve a smaller error in estimating $\beta^*$.
\begin{corro}\label{corro:refitting} For any $\delta \in (0,1)$, under the conditions of Theorem \ref{theo:constrained_permutation}, if it additionally holds that 
\begin{equation*}
\text{\emph{SNR}} > \frac{2 (1 + \sqrt{2})^2 \epss^{-4}}{\delta^2}\,\cdot \frac{k^2  \{w(\mc{T}) \vee \log n\}^2}{n}, \qquad n-k \geq 9 \vee 4d, 
\end{equation*}
the following events hold with probability at least $1 - \delta - 2\exp(-\frac{1}{2} (d \vee \log n))$:
\begin{align*}
\wh{S} \coloneq \{i:\;\wh{\Pi}_{ii} \neq 1 \} = S^*, \quad\qquad \nnorm{\wh{\wh{\beta}} - \beta^*}_2 \leq 
\frac{\sigma}{1 - \sqrt{\frac{4 d \vee \log n}{n - k}}} \sqrt{\frac{5 (d \vee \log n)}{n-k}}. 
\end{align*}
where $\wh{\wh{\beta}}$ denotes the ordinary least squares estimator based on data $\{(y_i, \mathbf{x}_i): \; i
\in [n] \setminus \wh{S} \}$.
\end{corro}
\noindent Under the conditions of the corollary, as long as $n - k = \Omega(n)$, the error rate in estimating $\beta^*$ is of the same order as if $\Pi^*$ were known in advance; the second term in \eqref{eq:bound_samples_unrelaxed} gets eliminated. It is important to note that support recovery (i.e., $\{\wh{S} = S^* \}$) is a weaker result than permutation recovery (i.e., $\{\Pi^* = \wh{\Pi} \}$). As discussed in more detail in $\S$\ref{subsec:permutationrecovery} below, the latter requires a considerably more stringent condition on the SNR than what is required in Corollary \ref{corro:refitting}.

\subsection{Convex relaxation}\label{subsec:convex_relaxation}    
While the approach of the previous section has appealing statistical properties, its computational hardness asks for computationally efficient alternatives with similar guarantees. As long as $\Pi^*$ is treated as a nuisance parameter, we may eliminate it in the following way. Introducing $f^* = \Pi^* X \beta^* - X \beta^* = (\Pi^* - I_n) X \beta^*$, model \eqref{eq:lrwp} can be re-expressed as 
\begin{equation*}
\y = X \beta^* + f^* + \eps.
\end{equation*}
This prompts the optimization problem 
\begin{align*}
&\min_{\beta \in \R^d, f \in \R^n} \, \nnorm{\y - X \beta - f}_2^2 \\
&\text{subject to} \;  f \in \bigcup_{\Pi \in \mc{P}_{n,k}} \text{range}(\Pi - I_n). 
\end{align*}
A first relaxation is given by 
\begin{align}\label{eq:relax_1} 
\begin{split}   
&\min_{\beta \in \R^d, f \in \R^n} \, \nnorm{\y - X \beta - f}_2^2  \\
&\text{subject to} \; \nnorm{f}_0 \leq k.   
\end{split}
\end{align}
We note in passing that one could additionally impose the constraint $\su f_i = 1$. However, it turns out that its addition does not yield significant statistical benefits, and it is thus omitted. 
Relaxation \eqref{eq:relax_1} is still not convex, but following the standard approach of replacing
the $\ell_0$-norm by the $\ell_1$-norm, we end up with the convex optimization problem 
\begin{align}\label{eq:relax_1b} 
\begin{split}
&\min_{\beta \in \R^d, f \in \R^n} \, \nnorm{\y - X \beta - f}_2^2  \\
&\text{subject to} \; \nnorm{f}_1 \leq b.   
\end{split}
\end{align}
Since it tends to be difficult to choose $b$ appropriately in practice, it is more convenient to work with the 
Lagrangian form of \eqref{eq:relax_1b}. After re-parameterizing $e = f / \sqrt{n}$\footnote{This reparameterization is merely for technical reasons} 
\begin{equation}\label{eq:relax_2}
 \min_{\beta \in \R^d, e \in \R^n} \frac{1}{n}\nnorm{\y - X \beta - \sqrt{n} e}_2^2 + \lambda \nnorm{e}_1, \quad \lambda > 0.
\end{equation}
Formulation \eqref{eq:relax_2} and variants thereof have been used in robust regression and signal recovery with gross corruptions (e.g., \cite{Laska2009, Nguyen2013, She2012, FoygelMackey}). In fact, \eqref{eq:relax_2} is equivalent to employing Huber's loss \cite{Huber1964} instead of squared loss, cf.~\cite{She2012} and the references therein. The connection to robust regression comes naturally as observations with mismatch between $\mathbf{x}$ and $y$ are likely to induce severe errors in model fitting beyond the usual noise, and hence take the role of outliers. Indeed, this reasoning could have been used to motivate \eqref{eq:relax_2} right away instead of via the sequence of relaxations stated above. Formulation \eqref{eq:relax_2} is related to least absolute deviation regression
\begin{equation}\label{eq:LAD}
 \min_{\beta \in \R^d} \nnorm{\y - X \beta}_1 
\end{equation}
in that \eqref{eq:LAD} is obtained from \eqref{eq:relax_2} in the limit $\lambda \rightarrow 0$ and the additional constraint
$\y = X \beta + \sqrt{n} e$. In that sense, \eqref{eq:LAD} can be seen as the counterpart to \eqref{eq:relax_2} in a noiseless setup. Problem \eqref{eq:LAD} has been analyzed under assumption \textbf{(G)} in the landmark papers of \cite{CandesTao2005, RudelsonVershynin2005} on the classical error correcting problem in coding theory.

Theorem \ref{theo:relaxation} below provides an upper bound on the $\ell_2$-error of the estimator of $\beta^*$ resulting as the optimal $\beta$ for \eqref{eq:relax_2}.
\begin{theo}\label{theo:relaxation} Let $(\wt{\beta}, \wt{e})$ be a minimizer of \eqref{eq:relax_2} with $\lambda = 4 (1 + M) \sigma \sqrt{2 \log(n) / n}$ for some $M > 0$. There exist constants $c_1, c_2, \epss > 0$ so that if
\begin{equation*}
k \leq c_1 \frac{n-d}{\log(n/k)},
\end{equation*}
the following holds with probability $\geq 1 - 2 \exp(-c_2 (n-d)) - 2n^{-M^2} - 2 \exp \left(-\frac{1}{2} (d \vee \log n) \right)$:
\begin{equation}\label{eq:bound_samples_relaxed}
\nnorm{\wt{\beta} - \beta^*}_2 \leq \frac{\sigma}{1 - \sqrt{\frac{4 d \vee \log n}{n}}} \left(\sqrt{\frac{5 (d \vee \log n)}{n}}  +  48 (1 + M) \frac{n}{n-d} \epss^{-1} \sqrt{\frac{2k \log n}{n}}\right).  
\end{equation}  
\end{theo}
\noindent Comparing the upper bounds \eqref{eq:bound_samples_unrelaxed} and \eqref{eq:bound_samples_relaxed} of the original problem and its relaxation, respectively, we observe a close agreement given that $w(\mc{T}) = \Theta(k \log(n/k))$. Apart from the slight change of the order in the second term, only constants differ. Similarly to Corollary \ref{corro:refitting}, we may consider a two-stage procedure in order to further improve upon \eqref{eq:bound_samples_relaxed}: given
$\wt{e}$ and a suitable threshold $t$, we let $\wt{S}_t = \{i: \, |\wt{e}_i| \geq t \}$ and obtain a plain least squares fit based on data $\{(y_i, \mathbf{x}_i): \; i
\in [n] \setminus \wt{S}_t \}$ yielding $\wt{\wt{\beta}}$. With $k$ assumed to be known, we may take $t$ as the $k$-th largest largest entry of $\wt{e}$ in absolute magnitude. The formal analysis is similar to Corollary \ref{corro:refitting} and is hence omitted.

\subsection{Estimation of $\Pi^*$ given an estimator of $\beta^*$}\label{subsec:permutationrecovery}
Having discussed estimation of $\beta^*$, we now come back to the problem of estimating $\Pi^*$. As indicated above, the latter problem turns out to be significantly more challenging than the former problem. In the sequel, we study the question of how the availability of an accurate estimate of $\beta^*$ can be leveraged to construct an estimator of $\Pi^*$ that is computationally feasible including the case $d > 1$. As mentioned earlier, except for $d=1$, the optimization problem in Eq.~\eqref{eq:lsq} 
is in general NP-hard. When $d=1$, that is there is only one predictor in the regression model, $\wh{\Pi}$ is determined as a minimizer of the optimization problem 
\begin{equation}\label{eq:permutation_d1}
\max_{\Pi \in \mc{P}_{n}} \{\scp{\Pi X}{\y},\scp{\Pi X}{-\y} \}  \qquad \text{subject to} \;\, d_H(\Pi, I_n) \leq k. 
\end{equation}
For $k = n$, maximizing each term inside the curly brackets is a specific instance of a linear assignment problem \cite{Burkard2009}, a class of problems that can be solved in polynomial time despite their combinatorial nature. In fact, it easy to see that
\begin{equation}\label{eq:maximuminnerproduct}
\max_{\Pi \in \mc{P}_{n}} \scp{\Pi X}{\y} = \su X_{(i)} \y_{(i)},
\end{equation}
where $X_{(i)}$ and $\y_{(i)}$ denote the $i$-th order statistic of $X$ and $\y$, respectively, $i \in [n]$. Hence, for $k = n$, problem
\eqref{eq:permutation_d1} reduces to two vanilla sorting operations. At this point, it is not well understood yet whether the
Hamming constraint $d_H(\Pi, I_n) \leq  k$ for general $k$ causes problem \eqref{eq:permutation_d1} to be NP-hard again. So far, we are not aware of any computationally efficient algorithm for general $k$.

Note that when $\beta^*$ is known, computing the least squares estimator of $\Pi^*$ reduces to solving precisely one of the two optimization problem encountered in \eqref{eq:permutation_d1} for $d = 1$:
\begin{equation}\label{eq:permutation_givenbetastar}
\max_{\Pi \in \mc{P}_{n}} \scp{\Pi X \beta^*}{\y}  \qquad \text{subject to} \;\, d_H(\Pi, I_n) \leq k. 
\end{equation}
At this point, a natural idea is to replace $\beta^*$ by an accurate and computationally feasible estimator like the one
discussed in $\S$\ref{subsec:convex_relaxation}. From a computational point of view, this already constitutes a simplification
as \eqref{eq:permutation_givenbetastar} reduces to an integer linear program; while problems in this class are still NP-hard in general, problem \eqref{eq:permutation_givenbetastar} can be considered as much more benign than the original problem \eqref{eq:lsq} which belongs to the class of quadratic assignment problems notorious for their computational hardness. Due to recent advances in integer programming that have meanwhile been taken advantage of for a series of other statistical problems \cite{Bertsimas2013, Bertsimas2016}, it turns out that problem \eqref{eq:permutation_givenbetastar} is practically feasible at least for $n$ in the order of a few thousands. For large $n$, we instead recommend estimating the support of $\Pi^*$ and then solve the unconstrained problem \eqref{eq:maximuminnerproduct} restricted to observations in the estimated support. Formally, denote by $\wt{S}$ an estimator of the support of $\Pi^*$ and let 
$\y_{\wt{S}}$ and $X_{\wt{S}}$ be the sub-vector of $\y$ and the row submatrix of $X$ corresponding to observations in $\wt{S}$, respectively.  We then estimate $\Pi^*$ by $\wt{\Pi}$ defined by
\begin{equation}\label{eq:permutation_estimated_support}
  \wt{\Pi} X =    \begin{pmatrix}
    \wt{\Pi}_{\wt{S},\wt{S}} \\
    I_{n - |\wt{S}|}
  \end{pmatrix} X,
\end{equation}
where 
\begin{equation*}
\wt{\Pi}_{\wt{S}} X = \begin{pmatrix}
  X_{\wt{S}} \\
  X_{\wt{S}^c}
\end{pmatrix}, \quad \text{and} \; \; \wt{\Pi}_{\wt{S}, \wt{S}} = \argmax_{\Pi \in \mc{P}_{|\wt{S}|}} \scp{\Pi X_{\wt{S}}}{\y_{\wt{S}}}.  
\end{equation*}
We now turn to the statistical limits of permutation recovery, including a lower bound on the SNR in the idealized case with known $\beta^*$. 
For simplicity, the theorem stated below is for the case $k = n$, but it generalizes to $k < n$ conditional on
having $\{\wt{S} = S^* \}$ in that all expressions in $n$ below would get replaced by $k$.
\begin{theo}\label{theo:permutationrecovery}
  For $b \in \R^d$, consider $\wh{\Pi}(b) = \argmax_{\Pi \in \mc{P}_n} \scp{\Pi X b}{\y}$.  
\begin{itemize}
\item[(a)] Let $\delta, \Delta, \eta > 0$, and let $\wh{\theta}$ be an estimator s.t.~the event $\{\nnorm{X \wh{\theta} - X\beta^*}_{\infty} \leq \sigma \Delta \}$ holds with probability at least $1 - \eta$. If 
  \begin{equation*}
    \dfrac{\|\beta^*\|_2^2}{\sigma^2}>\dfrac{n^2(n-1)^2}{4 \delta^2\pi} \left( \Delta + 2 \log\dfrac{n(n-1)}{\delta} \right)^2, \: \: \text{then} \: \: \p(\widehat{\Pi}(\wh{\theta}) \ne \Pi^*)\le 2\delta + \eta.
  \end{equation*}
\item[(b)] Suppose that $n \geq 5$. If
  \begin{equation*}
    \dfrac{\|\beta^*\|_2^2}{\sigma^2}< C n^2, \: \: \text{for some constant $C>0$, then} \: \: \p(\widehat{\Pi}(\beta^*)\ne \Pi^*)\ge .35.
  \end{equation*}
\end{itemize}
\end{theo}
\noindent If $\beta^*$ is known, specializing part (a) to the case $\Delta = 0$ asserts that $\text{SNR} = \wt{\Omega}(n^4)$ is a sufficient condition for permutation recovery. This is to be compared to a result in \cite{Pananjady2016} stating that the estimator
$\wh{\Pi}$ in \eqref{eq:lsq} recovers $\Pi^*$ if $\text{SNR} = \Omega(n^c)$, where the constant $c > 0$ is not specified. Next, let us now turn to the case where $\beta^*$ is substituted by an estimator $\wh{\theta}$. Using standard concentration arguments, one shows that
\begin{equation*}
  \nnorm{X \wh{\theta} - X\beta^*}_{\infty} \leq \max_{1 \leq i \leq n} \nnorm{\x_i}_2 \nnorm{\wh{\theta} - \beta^*}_2 \leq
  C (\sqrt{d} + \sqrt{\log n}) \nnorm{\wh{\theta} - \beta^*}_2
\end{equation*}
holds with high probability. Specifically, considering the estimator $\wt{\beta}$ in $\S$\ref{subsec:convex_relaxation}, Theorem \ref{theo:relaxation} then yields $\nnorm{X \wt{\beta} - X\beta^*}_{\infty} \leq C \sigma \sqrt{n}$ for $k$ small enough, with high probability. Inserting $\Delta = C \sqrt{n}$ into part (a) then results into the requirement $\text{SNR} = \wt{\Omega}(n^5)$. We stress again that as opposed to $\wh{\Pi}$, the estimator $\wt{\Pi}(\wt{\beta})$ is computationally appealing as it is obtained from a quadratic program \eqref{eq:relax_2} and subsequent sorting.

Finally, part (b) provides evidence that the condition $\text{SNR} = \Omega(n^2)$ is necessary for permutation recovery. While the result in (b) concerns a specific estimator, there does not seem to be much indication for the existence of a substantially better estimator. In \cite{Pananjady2016}, it is shown that $\text{SNR} = \Omega(n^{8/9})$ is necessary for \emph{any} estimator.

\section{Numerical results}\label{sec:numericalresults}

\subsection{Simulated data}

Below, we present the results of two synthetic data experiments that are meant to serve as illustration of the developments in the previous section. In the first set of experiments, we generate $n = 200$ observations from model \eqref{eq:lrwp} under $(\textbf{G})$ with $d = 10$,
$\sigma \in \{.01, .02, .05, .1, .2, .5, 1\}$, and $k/n \in \{.01, .02, .05, .1, .15, .2, \ldots, .5\}$, where the support of $\Pi^*$ is selected uniformly at random. The parameter $\beta^*$ is generated from the uniform distribution on $\mathbb{S}^{d-1}$. For each pair $(\sigma, k/n)$, we conduct 100 independent replications.  

We compare the following estimators:
\begin{itemize}
\item[(i)] The naive estimator (ordinary least squares estimator) of $\beta^*$ that ignores the fact that a fraction of the data is mismatched; this corresponds to the choice $k = 0$ in \eqref{eq:lsq}.
\item[(ii)] The estimator $(\wt{\beta}, \wt{e})$ of $\S$\ref{subsec:convex_relaxation} with the choice $\lambda = 0.2 \sigma \sqrt{\log(n) / n}$.
\item[(iii)] A thresholded version of the estimator in (ii) that discards those observations
             corresponds to the largest $k$ elements of $\{|\wt{e}_i|, \; i \in [n]\}$, and performs a least squares re-fit using the remaining $n-k$ observations.   
\end{itemize}
The estimator $(\wt{\beta}, \wt{e})$ is computed in \texttt{CVX} \cite{cvx}.  

The estimator $\wh{\beta}$ in $\S$\ref{subsec:sparsePi} is compared to $\wt{\beta}$ in a second set of simulations only for the case $d = 1$. In that case,  computation of $\wh{\beta}$ reduces to \eqref{eq:permutation_d1} where each of the two inner optimization problems can be cast as an integer linear program with $n^2$ variables and  $2n + 1$ linear constraints. It turns out that the general purpose routine \texttt{cplexbilp} in \texttt{IBM CPLEX} \cite{CPLEX} is able to solve such problems in only a few seconds for $n = 200$. Again, such reduction is limited to the case $d = 1$. 

\begin{figure}[h]
\begin{center}
\hspace*{-1ex}\begin{tabular}{ccc}
(i)    &    (ii)     & (iii)    \\
\includegraphics[width = 0.32\textwidth]{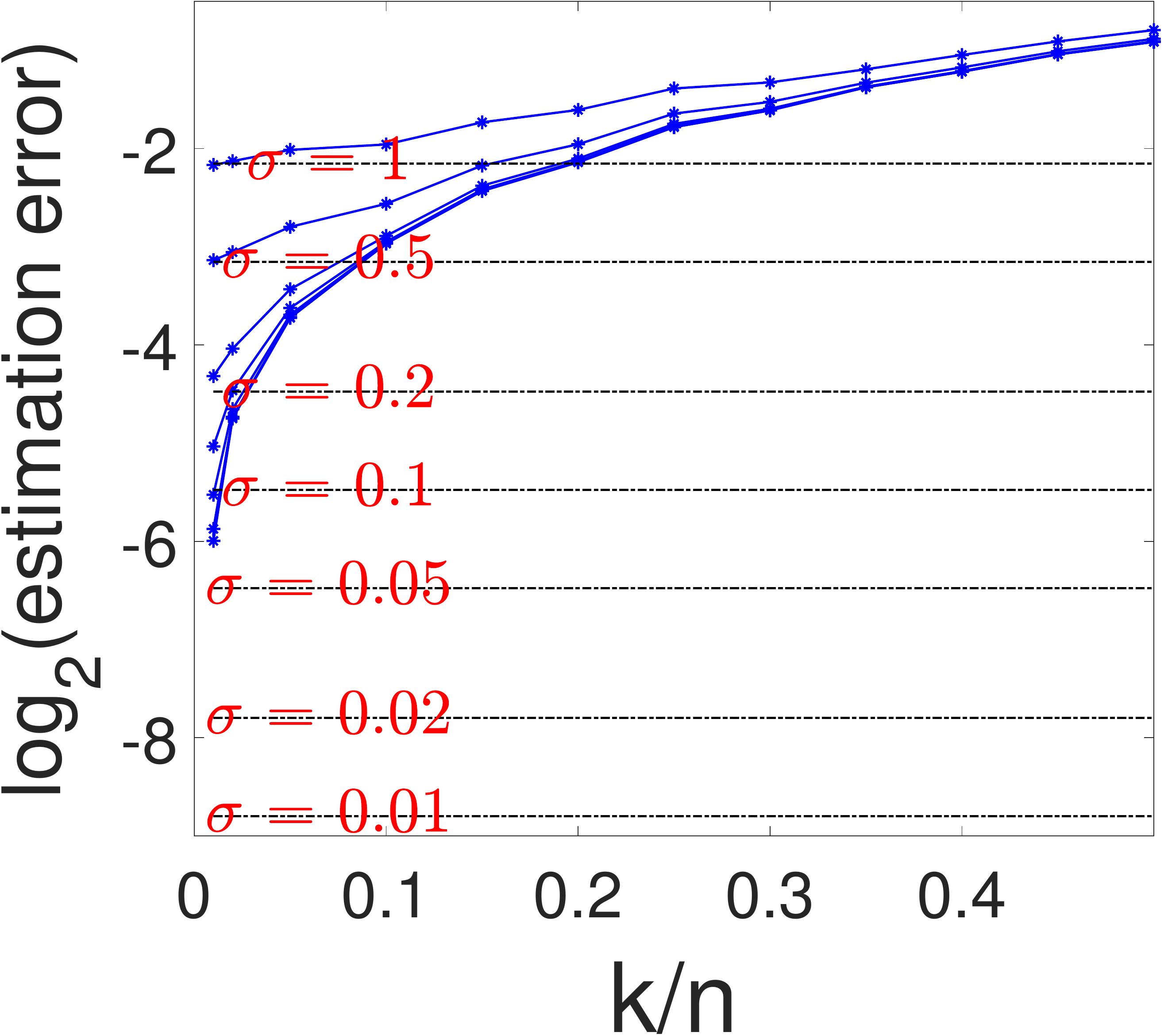} & \includegraphics[width = 0.32\textwidth]{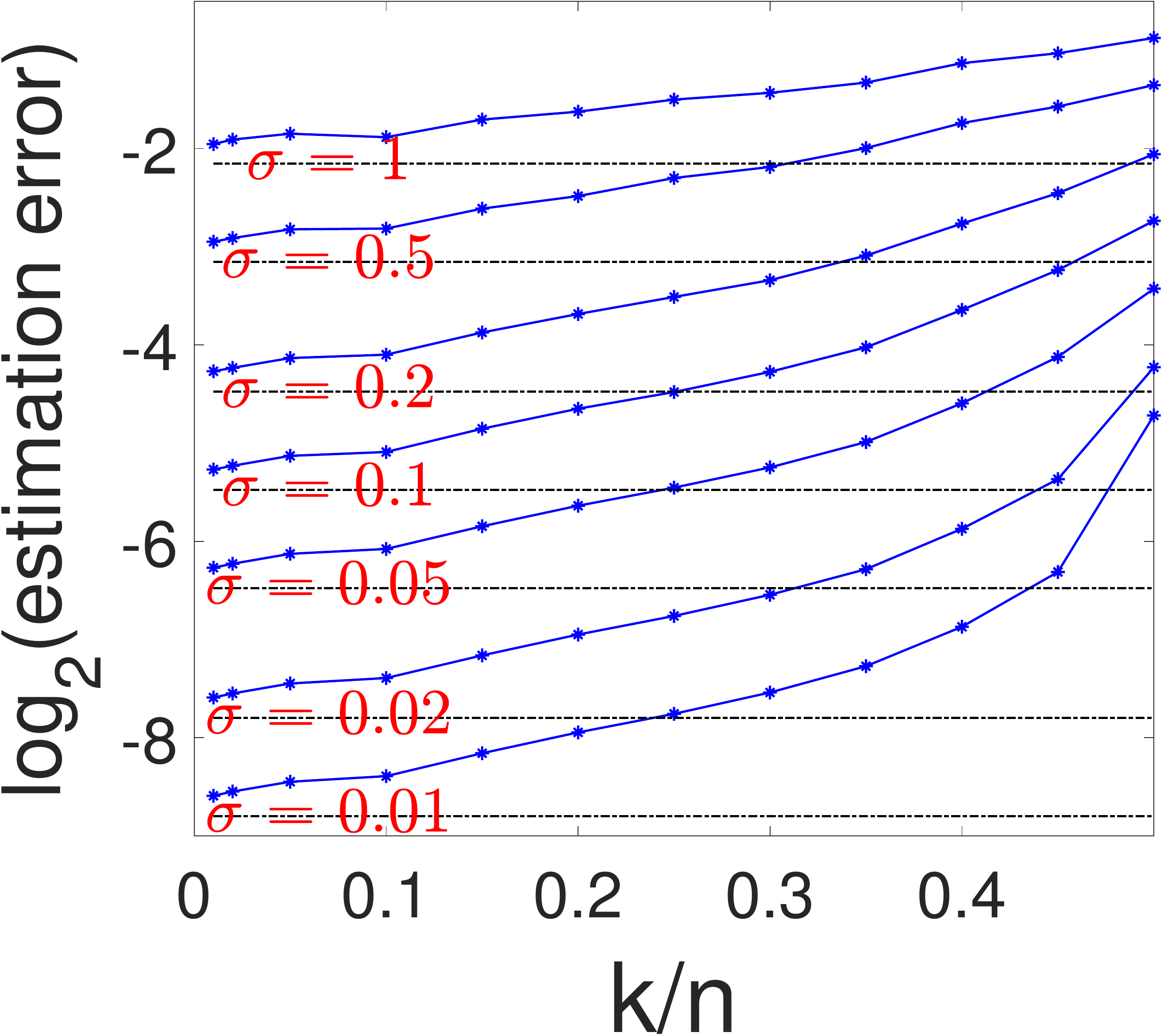} & \includegraphics[width = 0.32\textwidth]{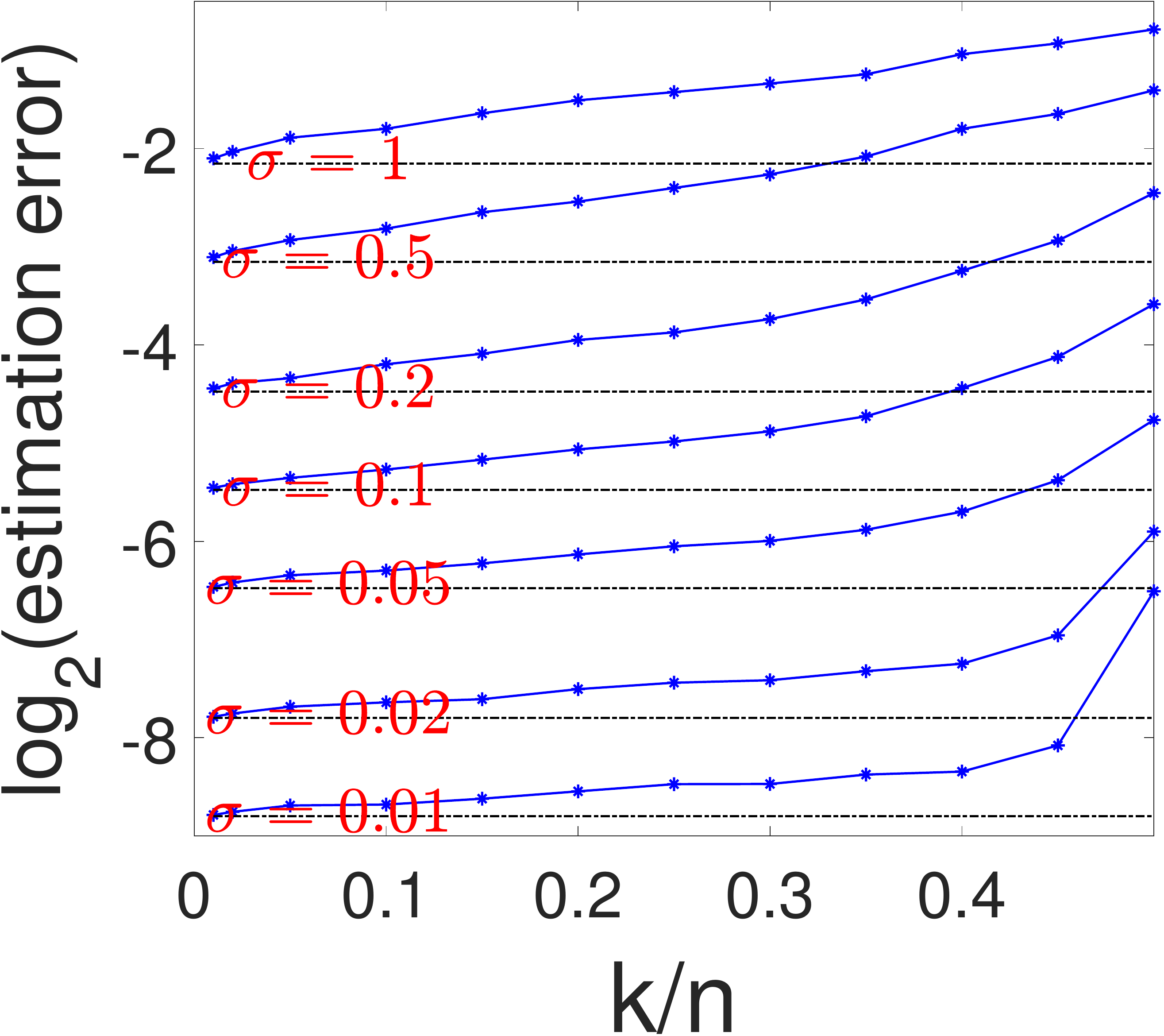} 
\end{tabular}
\caption{Comparison of the average $\ell_2$-errors ($\log_2$-scale) in recovering $\beta^*$ for the approaches (i) to (iii) described above. Each curve corresponds to a different value of the noise level $\sigma$. The black dashed lines correspond to the oracle (ordinary least squares with knowledge of $\Pi^*$).}\label{fig:simulations_1}
\end{center}
\end{figure}
Inspection of Figure \ref{fig:simulations_1} shows that the approach (ii) improves dramatically over the naive estimator (i) as long as the SNR is not too small (for $\sigma = 1$ and $\sigma = 0.5$, there is no longer a visible improvement). The results look promising in that the tolerable fraction of permuted observations can be as much as $0.5$ as long as the noise level is small; "tolerable" here refers to the fact that the estimation error increases gently with the fraction of permuted observations as opposed to (i) with a sharp increase in error as $k/n$ moves away from zero. Approach (iii) appears to yield further improvements over (ii) for large SNR. For small SNR, (ii) and (iii) are not distinguishable. This observation aligns well with Corollary \ref{corro:refitting}. 
\begin{figure}[h]
\begin{center}
\begin{tabular}{cc}
$\wh{\beta}$     &     $\wt{\beta}$         \\
\includegraphics[width = 0.4\textwidth]{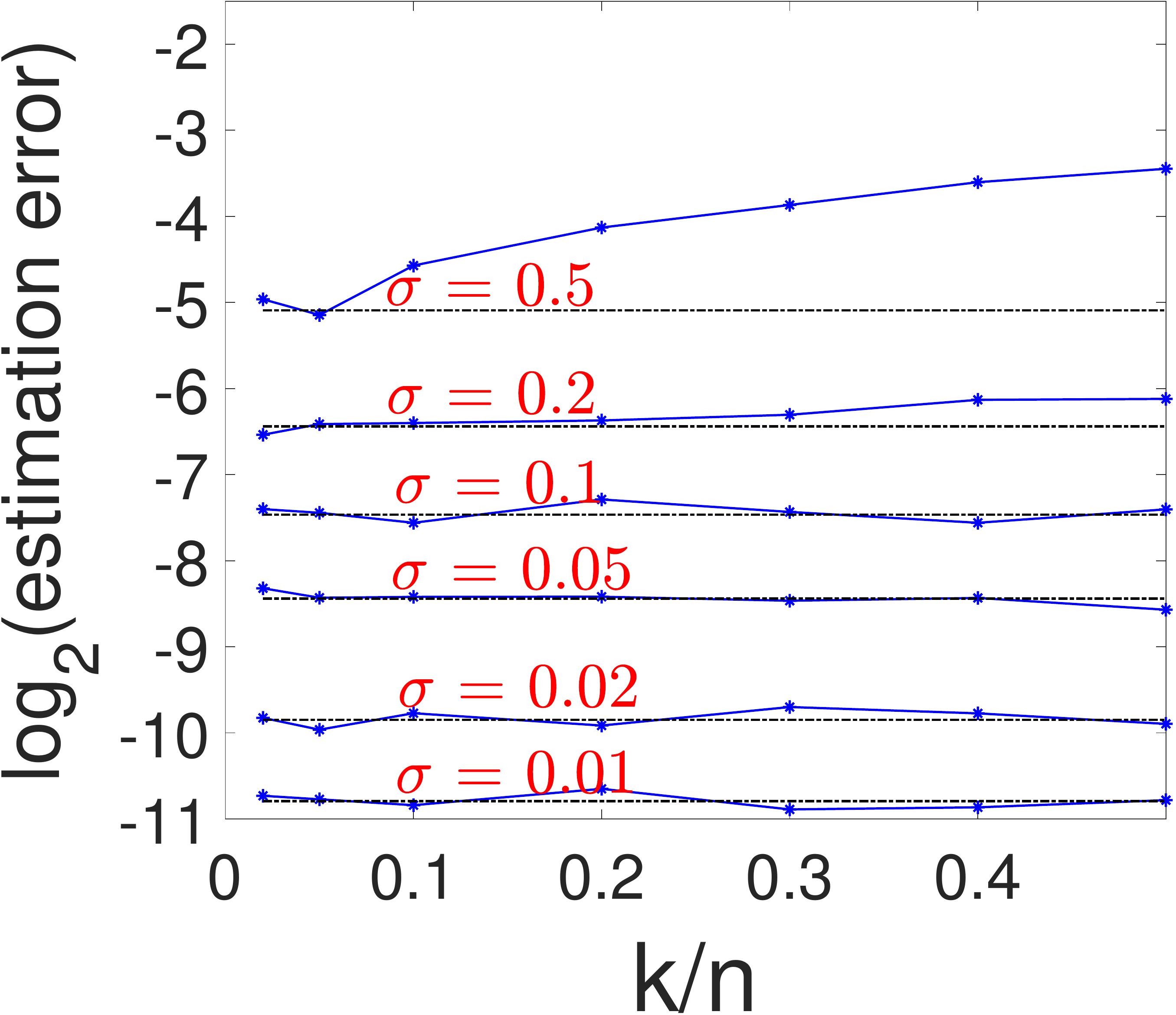} \hspace*{3ex} & \hspace*{3ex} \includegraphics[width = 0.4\textwidth]{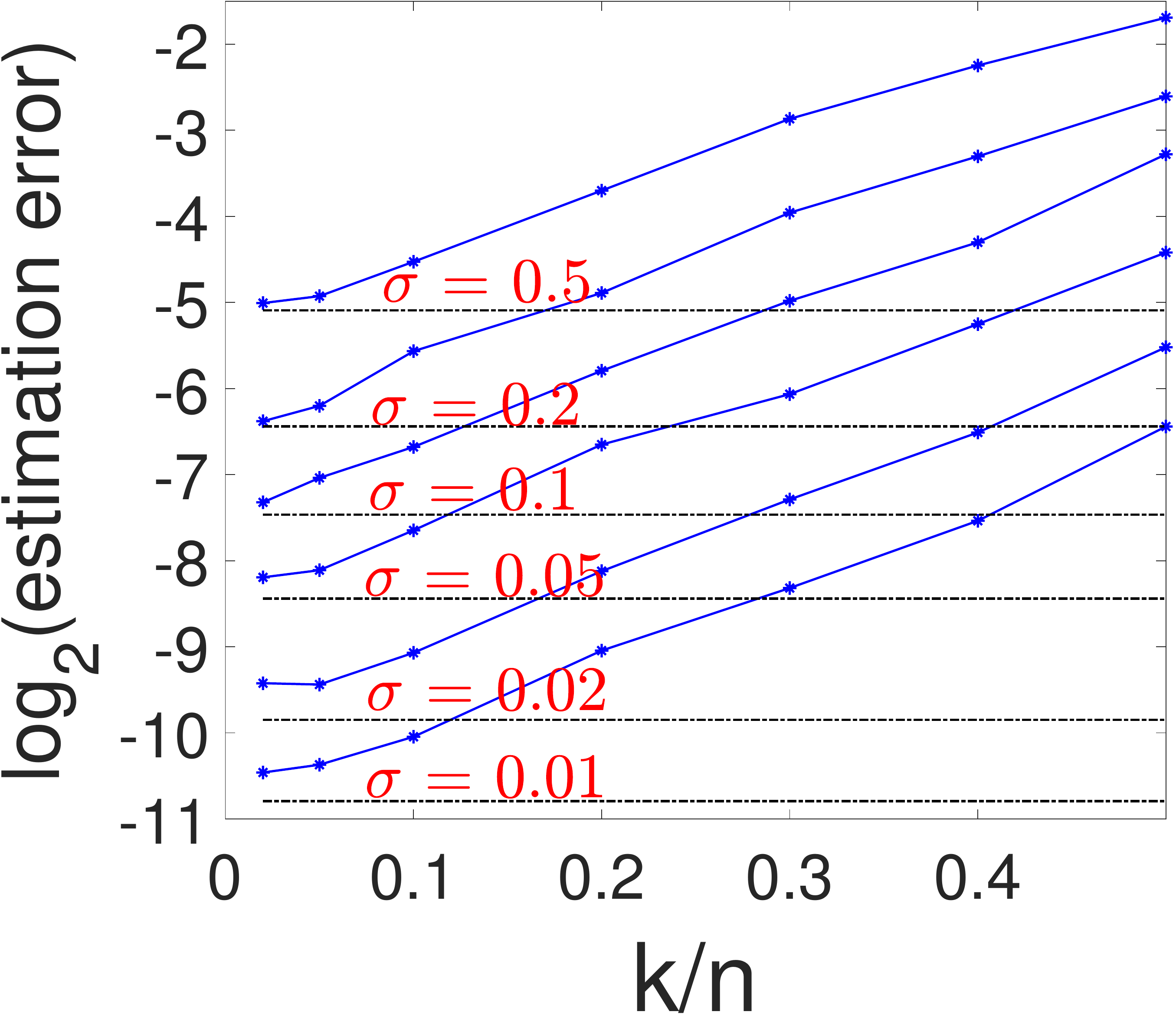} 
\end{tabular}
\caption{Comparison of the average $\ell_2$-errors ($\log_2$-scale) in recovering $\beta^*$ for the estimators $\wh{\beta}$ and $\wt{\beta}$ for $d = 1$. Each curve corresponds to a different value of the noise level $\sigma$. The black dashed lines correspond to the oracle (least squares with knowledge of $\Pi^*$).}\label{fig:simulations_2}
\end{center}   
\end{figure}
Figure \ref{fig:simulations_2} reveals that for $d = 1$, the estimator $\wh{\beta}$ performs substantially better than the convex formulation yielding $\wt{\beta}$. We observe that $\wh{\beta}$ is on par with the oracle for $\sigma < 0.2$, being hardly effected by an increase in $k/n$. That starts changing for $\sigma = 0.2$, while for $\sigma = 0.5$, the adverse effect of an increase in $k/n$ becomes clearly visible. At this point, it is not clear whether similar improvements of 
$\wh{\beta}$ over $\wt{\beta}$ would be observed for $d > 1$. The analysis in $\S$\ref{sec:mainresults} only provides (comparable) upper bounds for both approaches, and lower bounds would be needed to answer that question. The observation made from Figure \ref{fig:simulations_2} hence points towards an important open direction.      

\subsection{Real data}\label{sub:real}

The El Ni\~no data set in the UCI Machine Learning Repository \cite{UCI} contains oceanographic and surface meteorological readings taken from a series of buoys positioned throughout the equatorial Pacific. The data set consists of $n = 93,935$ records with the following attributes: buoy identifier, date, location (\texttt{latitude} and \texttt{longitude}), zonal and meridional wind speeds (\texttt{zon}  and \texttt{mer}), relative humidity (\texttt{humidity}), air temperature (\texttt{air.temp}), sea surface temperature and subsurface temperatures down to a depth of 500 meters (\texttt{s.s.temp}). The following linear regression model is considered:
\begin{align}\label{eq:elnino}
\begin{split}
\texttt{air.temp} & = \beta_0 + \beta_{\texttt{zon}} \cdot \texttt{zon} + \beta_{\texttt{mer}} \cdot \texttt{mer} +   \\
                  &\;\qquad + \beta_{\texttt{humidity}} \cdot \texttt{humidity}  +  \beta_{\texttt{s.s.temp}} \cdot \texttt{s.s.temp} + \epsilon.
\end{split}
\end{align}
The results of least squares regression 
indicate an excellent fit of this model with an R-squared of about $0.9$. With the goal to mimic the situation of data merging
based on non-unique identifiers, the data set is divided into two data sets A and B, with A containing the response variable and B containing the predictors. The variables \texttt{latitude} and \texttt{longitude} are maintained in both data sets, and used as quasi-identifiers for record linkage of A and B with the \texttt{R} package \texttt{fastLink} \cite{fastlink}. Since (\texttt{latitude}, \texttt{longitude})-pairs are generally associated with multiple observations, the linkage process is not free of mismatches. The merged data set in turn follows the format depicted in Figure \ref{fig:permutationmodel} with 
a permutation affecting the correspondence between responses and predictors. As shown in Table \ref{tab:relax}, least squares regression with the thus merged data set leads to an increase of the residual sum of squares by $27\%$. In addition, the estimates of the regression parameters of the linear model \eqref{eq:elnino} change noticeably. As an alternative, we consider the approach \eqref{eq:relax_2} studied in detail herein, where the parameter $\lambda$ is chosen as 
$\lambda = 2 c \wh{\sigma} / \sqrt{n}$ with $c = 1.345$ and $\wh{\sigma}$ being a robust estimator of the residual standard error. More specifically, we use the function \texttt{rlm} (short for ``robust linear model fitting'') in the \texttt{R} package \texttt{MASS} \cite{MASS} with its default arguments which is in exact correspondence to the above choice of $\lambda$, as follows from the connection between linear regression with the Huber loss and formulation \eqref{eq:relax_2} (cf., e.g., Proposition 3.1 in \cite{She2012}). The figures 
in Table \ref{tab:relax} indicate that this approach provides some remedy relative to the naive use of the least squares estimator based on the merged data set. The estimates of the regression parameters from \eqref{eq:relax_2} reduce the gap to those obtained with the original data set by about a factor of one half in terms of $\ell_2$-distance. Moreover, approach \eqref{eq:relax_2} also yields a better fit in terms of the mean squared prediction error on a collection of twenty hold-out sets obtained from leaving out different random subsets of 10\% of the observations; in each of those twenty runs, the hold-in data set is split and subsequently re-linked in the same manner as described above. The reduction in excess risk (relative to the least squares estimator) achieved by \eqref{eq:relax_2} in comparison to the naive least squares solution is again a factor of roughly one half.


\begin{table}
\begin{center}
\begin{tabular}{|l|c|c|c|c|c|c|c|c|}
\hline
                           &  {\footnotesize \texttt{interc}}  & {\footnotesize \texttt{zon}}  & {\footnotesize \texttt{mer}}   & {\footnotesize \texttt{humidity}}   & {\footnotesize \texttt{s.s.temp}}  & 
 {\small RMSE} & {\small $\ell_2$-dist.} & {\small $\emptyset$-hold-out error} \\[.35ex]
\hline
 &          &          &                      &                     &                    &                     & & 
\\[-1.5ex]
{\small $\wh{\beta}^{\text{oracle}}$} & {\small 5.15}  & {\small $-$.056}   &  {\small$-$.031} &  {\small$-$.022} &  {\small .844}   & {\small .509}  & 0& {\small .260 (4.6$\cdot10^{-3}$)}
\\[.7ex]
{\small $\wh{\beta}^{\text{ols}}$}    & {\small 6.72} & {\small $-$.037} &     {\small$-$.045}   &   {\small$-$.017}   &     {\small .774}  &  {\small .771}                    &  {\small 1.57} & {\small .276 (4.9$\cdot10^{-3}$)} \\[.7ex]

{\small $\wt{\beta}$}              &  {\small  5.74} & {\small $-$.044}   &   {\small$-$.037}   &       {\small $-$.016}     &    {\small .806}   & {\small .773}  &    {\small .59} & {\small.267 (4.8$\cdot10^{-3}$)}  \\    
\hline               
\end{tabular}
\end{center}
\vspace*{-0.2in}
\caption{Results of the real data analysis. Here, $\wh{\beta}^{\text{oracle}}$ denotes the least squares estimator based on the original data set, $\wh{\beta}^{\text{ols}}$ denotes the naive least squares estimator, and $\wt{\beta}$ denotes the estimator from formulation 
\eqref{eq:relax_2}. The columns labeled \texttt{zon}, $\ldots$, \texttt{s.s.temp} contain the regression parameter estimates for the respective variable; \texttt{interc} refers to the intercept. The three rightmost columns contain the root mean square errors (RMSEs) on the original ($\wh{\beta}^{\text{oracle}}$) respectively the merged data set ($\wh{\beta}^{\text{ols}}$,  $\wt{\beta}$),  the $\ell_2$-distance to $\wh{\beta}^{\text{oracle}}$, and the average hold-out errors (standard errors in parentheses), respectively.}\label{tab:relax}
\end{table}

\section{Conclusion}\label{sec:conclusion}
Regression problems in which the correspondence between responses and predictors has been lost can be traced back to classical work in statistics and naturally arises in the context of record linkage. Despite its long history, the problem has seen revived interest lately and has been analyzed through the lens of modern non-asymptotic statistical theory. An additional layer of complexity to be dealt with is the computational hardness of (quadratic) assignment problems. In the present paper, both aspects are taken into account under the assumption of sparsely permuted data. One central insight of our work is that the availability of an accurate estimator of $\beta^*$ may be helpful in circumventing the computational barrier associated with permutation recovery. On the other hand, preliminary numerical results for $d = 1$ indicate a gap in statistical performance between a presumably hard formulation capturing precisely the notion of a sparse permutation and a computationally convenient relaxation that introduces a loss of information. This observation deserves further investigation. 

\subsection*{Acknowledgments}

The second author would like to thank William Winkler for his guidance and Anindya Roy for helpful discussions related to this paper.



\bibliographystyle{siam}
{
\bibliography{references_M.bib}
}
\appendix
\section{Proof of Proposition \ref{prop:negative_result}}

Consider the least squares problem \eqref{eq:lsq} with $\y = \eps$:
\begin{equation*}
\min_{\Pi, \beta} \nnorm{\eps - \Pi X \beta}_2^2. 
\end{equation*} 
Omitting the term not depending on $\beta$, an equivalent problem is given by 
\begin{equation*}
\min_{\Pi, \beta} -2 \scp{\Pi \eps}{X \beta} + \nnorm{X \beta}_2^2. 
\end{equation*}
Re-parameterizing $X \beta = u \cdot r$, with $r \geq 0$ for $u \in \text{range}(X) \cap \mathbb{S}^{n-1}$, we obtain the optimization problem
\begin{equation*}
\min_{\Pi, u \in \text{range}(X) \cap \mathbb{S}^{n-1}, r \geq 0}  -2 \scp{\Pi \eps}{u} r + r^2. 
\end{equation*}
Observe that for any $u$ and $\Pi$, the optimal value of $r$ is found as 
\begin{equation*}
\wh{r}(u) = \max\{\scp{\Pi \eps}{u},0 \}. 
\end{equation*}
In particular, for a minimizer $(\wh{\Pi}, \wh{u})$,
\begin{equation*}
\wh{r} \coloneq \wh{r}(\wh{\Pi}, \wh{u}) = \nscp{\wh{\Pi} \eps}{\wh{u}}
\end{equation*}
so that $\wh{r} = \nnorm{X \wh{\beta}}_2$, and the optimal value of the objective is given by 
$-\wh{r}^2$. We hence have  
\begin{align*}
\wh{r} &= \max_{u \in \text{range}(X) \cap \mathbb{S}^{n-1}, \Pi} \nscp{\Pi \eps}{u} \\
       &\geq  \max_{1 \leq j \leq d, \Pi} \nscp{\Pi \eps}{X_j / \nnorm{X_j}_2} \\
       &\geq  \max_{\Pi} \nscp{\Pi \eps}{X_1} \frac{1}{\nnorm{X_1}_2}. 
\end{align*}
We now bound
\begin{equation*}
\max_{\Pi} \nscp{\Pi \eps}{X_1} \geq \nscp{\Pi_{\pm} \eps}{X_1},
\end{equation*}
where the permutation $\Pi_{\pm}$ is defined as follows. Define 
\begin{align*}
n_{\eps,+} = \{i:\; \epsilon_i > 0\}, \quad n_{\eps,-} = n - n_{\eps,+},  \quad n_{X_1,+} = \{i:\; X_{i1} > 0\}, \quad n_{X_1,-} = n - n_{X_1,+}.
\end{align*}
Let further 
\begin{equation}\label{eq:nplusnminus}
n_+ = \min\{n_{\eps,+},  n_{X_1,+} \},  \quad n_- = \min\{n_{\eps,-},  n_{X_1,-} \}, \quad 
n_{\pm} = n - n_+ - n_-. 
\end{equation}
Denote by 
\begin{alignat*}{2}
&i_1^+ < \ldots < i_{n_{X,+}}^+ \quad &&\text{the indices of the positive entries of $X_1$},\\
&j_1^+ < \ldots < j_{n_{\eps,+}}^+ \quad &&\text{the indices of the positive entries of $\eps$},\\
&i_1^- < \ldots < i_{n_{X,-}}^- \quad &&\text{the indices of the negative entries of $X_1$},\\
&j_1^- < \ldots < j_{n_{\eps,-}}^- \quad &&\text{the indices of the negative entries of $\eps$}.
\end{alignat*}
Furthermore, let 
\begin{align*}
&i_{1}^{\pm} < \ldots < i_{n_{\pm}}^{\pm} \\
&j_{1}^{\pm} < \ldots < j_{n_{\pm}}^{\pm} 
\end{align*}
be the indices in 
\begin{align*}
                    & \{1,\ldots,n\} \setminus (\{i_1^+,\ldots, i_{n_+}^+\} \cup \{i_1^-,\ldots, i_{n_-}^-\}) \\
\text{respectively}\;\, & \{1,\ldots,n\} \setminus (\{j_1^+,\ldots, j_{n_+}^+\} \cup \{j_1^-,\ldots, j_{n_-}^-\})
\end{align*}
in increasing order. We then construct $\Pi_{\pm}$ such that 
\begin{equation*}
\nscp{\Pi_{\pm} \eps}{X_1} = \sum_{k = 1}^{n_+} X_{1{i_k^+}} \epsilon_{j_k^+} + 
\sum_{k = 1}^{n_-} X_{1{i_k^-}} \epsilon_{j_k^-} + \sum_{k = 1}^{n_{\pm}} X_{1{i_k}^{\pm}} \; \epsilon_{{j_k}^{\pm}}. 
\end{equation*}
All terms inside the first two sums are i.i.d.~from the same distribution as $|g| |h|$, where $g$ and $h$ are independent
$N(0,1)$-random variables, and accordingly the terms inside the last summand are i.i.d.~from the same distribution as
$-|g||h|$. Conditional on $n_{\pm}$,  hence $\nscp{\Pi_{\pm} \eps}{X_1}$ follows the same distribution as 
\begin{equation*}
\sum_{i = 1}^{n - n_{\pm}} |g_i| |h_i|  - \sum_{i=n - n_{\pm} + 1}^{n} |g_i| |h_i|,
\end{equation*}
where $\{ g_i \}_{i = 1}^n$ and $\{ h_i \}_{i = 1}^n$ are two independent sequences of $n$ i.i.d.~random variables from the $N(0,1)$-distribution. The expectation of the above expression (conditional on $n_{\pm}$)
is given by 
\begin{equation*}
(n - 2 n_{\pm}) \E[|g||h|] = (n - 2 n_{\pm}) \E[|g|]^2 = (n - 2 n_{\pm}) \frac{2}{\pi}. 
\end{equation*}
From Proposition 5.16 in \cite{Vershynin2010}, we have the following concentration inequality: 
\begin{equation*}  
\p \left(\left|\sum_{i = 1}^{n - n_{\pm}} (|g_i| |h_i| - \E[|g||h|]) - \sum_{i=n - n_{\pm} + 1}^{n} (|g_i| |h_i| - \E[|g||h|]) \right| \geq 
(n - 2 n_{\pm}) \frac{1}{\pi} \Big| n_{\pm} \right) \leq \exp(-n c).
\end{equation*}
Hence, conditional on $n_{\pm}$, with probability at least $1 - \exp(-nc)$, 
\begin{equation*}
\sum_{i = 1}^{n - n_{\pm}} |g_i| |h_i|  - \sum_{i=n - n_{\pm} + 1}^{n} |g_i| |h_i|  \geq \frac{1}{\pi} (n - 2 n_{\pm}). 
\end{equation*}
Consider the definition of $n_+$ and $n_-$ Observe that $n_{\eps, +}$, $n_{X_1, +}$, $n_{\eps, -}$ and $n_{X_1, -}$ are each binomial random variables with $n$ trials and probability of success $1/2$. By repeated application of Hoeffding's inequality, we hence have 
\begin{equation*} 
\p(n_{+} \leq 3n/8) \leq 2\exp(-n/32), \quad \p(n_{-} \leq 3n/8) \leq 2 \exp(-n/32).
\end{equation*}
Conditional on the events $\{ n_{+} > 3n/8\}$ and $\{ n_{-} > 3n/8\}$, we have $n_{\pm} \leq 1/4$ and thus with probability at least $1 - C \exp(-nc)$
\begin{equation*}
\sum_{i = 1}^{n - n_{\pm}} |g_i| |h_i|  - \sum_{i=n - n_{\pm} + 1}^{n} |g_i| |h_i|  \geq \frac{n}{2\pi} . 
\end{equation*}
Accordingly, the same applies to the event $\{ \nscp{\Pi_{\pm} \eps}{X_1} > \frac{n}{2\pi} \}$. 
Lemma \ref{lem:normconcentration} in Appendix \ref{app:aux} yields $\nnorm{X_1}_2 \leq 2 \sqrt{n}$ with probability at least $1 - \exp(-n/2)$. Putting together the pieces, it follows that with probability at least $1 - C \exp(-nc)$
\begin{equation*}
\nnorm{X \wh{\beta}}_2 = \wh{r} \geq \frac{\sqrt{n}}{4\pi} \quad \Rightarrow \; \nnorm{\wh{\beta}}_2 \geq  \frac{\sqrt{n}}{\sigma_{\max}(X) 4\pi},
\end{equation*}   
where $\sigma_{\max}(X)$ denotes the maximum singular value of $X$. Finally, Lemma \ref{lem:extremesingularvalues} in Appendix \ref{app:aux} yields that $\sigma_{\max}(X) \leq \sqrt{2n} + \sqrt{d}$ with probability at least $1 - \exp(-c n)$, which concludes the proof.

\section{Proof of Theorem \ref{theo:constrained_permutation}}\label{app:theo:constrained_permutation}

We start with a basic, yet important observation that allows us to decouple $\wh{\Pi}$ and $\wh{\beta}$. Let
$\pr_X$ and $\pr_X^{\perp}$ denote the orthogonal projection on $\text{range}(X)$ respectively its orthogonal complement.

\begin{lemma}\label{lem:decoupling} Consider optimization problem \eqref{eq:lsq}. Then, if $n \geq  d$, with probability one 
\begin{equation}\label{eq:decomposition}  
  \wh{\Pi}^{\T} \in \argmin_{\Pi \in \mc{P}_{n,k}} \nnorm{\pre_X^{\perp} \Pi \mathbf{y}}_2^2, \qquad \wh{\beta} \in \left\{ \left(\frac{X^{\T} X}{n} \right)^{-1} \frac{X^{\T} \wh{\Pi}^{\T} \y}{n},
  \; \, \wh{\Pi}^{\T} \in \argmin_{\Pi \in \mc{P}_{n,k}} \nnorm{\pre_X^{\perp} \Pi \mathbf{y}}_2^2 \right \}.  
\end{equation}
\end{lemma}

\begin{proof} We have 
\begin{align*}
  \min_{\Pi, \beta} \|\Pi X\beta-\y\|_2^2 &= \min_{\Pi, \beta} \|X\beta-\Pi \y\|_2^2  \\
                                        &=  \min_{\Pi, \beta} \left\{ \|\pr_X^{\perp} \Pi \y\|_2^2   +  \|X\beta- \pr_X \Pi \y\|_2^2 \right\}. 
\end{align*}
The second part can always be made equal to zero by choosing $\beta$ such that $X \beta = \pr_X \Pi \y$. It hence suffices to
minimize the first part w.r.t.~$\Pi$ and to back-substitute the result into the second part. Under model (\text{\textbf{G}}), $X$ is
non-singular with probability one as long as $n \geq d$ which yields the second expression in \eqref{eq:decomposition}. 
\end{proof}
\noindent In the next lemma, we bound the $\ell_2$-distance between $\Pi^{*\T} \y$ and $\wh{\Pi}^{\T} \y$. 
\begin{lemma}\label{lem:bound_pi} Under the conditions of Theorem \ref{theo:constrained_permutation}, the following holds
  with probability at least $1 - \frac{7}{2} \exp \left(-\frac{\psi(\epss, d/n) \nu_n^2}{8} \right) - \exp(-\{ w(\mc{T}) \vee \log n\})$,
  where $\psi(\epss, d/n)$ is defined in \eqref{eq:lowerbound} below. 
\begin{equation*} 
\nnorm{\wh{\Pi}^{\T} \y - \Pi^{*\T} \y}_2 \leq  \frac{2 (1 + \sqrt{2})\sigma \epss^{-2} \{w(\mc{T}) \vee \log n\}}{\sqrt{n}}.
\end{equation*}
\end{lemma}
\begin{proof} In view of Lemma \ref{lem:decoupling}, by definition of $\wh{\Pi}^{\T}$ and the fact that $\Pi^{*\T}$ is a feasible solution, we have
 \begin{align*} 
&\nnorm{\pr_X^{\perp}\wh{\Pi}^{\T} \y}_2^2  \leq  \nnorm{\pr_X^{\perp}\Pi^{*\T} \y}_2^2 \\
  \Rightarrow \quad & \nnorm{\pr_X^{\perp}(\wh{\Pi}^{\T}  - \Pi^{*\T}) \y + \pr_X^{\perp}\Pi^{*\T} \y }_2^2 \leq \nnorm{\pr_X^{\perp}\Pi^{*\T} \y}_2^2 \\
  \Rightarrow \quad  & \nnorm{\pr_X^{\perp}(\wh{\Pi}^{\T}  - \Pi^{*\T}) \y}_2^2 \leq 2 |\nscp{\pr_X^{\perp}\Pi^{*\T} \y}{(\wh{\Pi}^{\T}  - \Pi^{*\T}) \y}|. 
 \end{align*}
Since
\begin{equation*}
\pr_X^{\perp} \Pi^{*\T} \y = \pr_X^{\perp} (X \beta^* + \eps) = \pr_X^{\perp}\eps,
\end{equation*}
we have the implication that
\begin{equation*}
  \nnorm{\pr_X^{\perp}(\wh{\Pi}^{\T}  - \Pi^{*\T}) \y}_2^2 \leq 2 |\nscp{\pr_X^{\perp}\eps}{(\wh{\Pi}^{\T}  - \Pi^{*\T}) \y}|.
\end{equation*}
Dividing both sides by $\nnorm{\pr_X^{\perp}(\wh{\Pi}^{\T}  - \Pi^{*\T}) \y}_2$, we obtain that
\begin{align}
  &\nnorm{(\wh{\Pi}^{\T}  - \Pi^{*\T}) \y}_2^2 \nnorm{\pr_X^{\perp} \wh{\Delta}}_2^2 \leq 2 |\nscp{\pr_X^{\perp}\eps}{\wh{\Delta}}| \nnorm{(\wh{\Pi}^{\T}  - \Pi^{*\T}) \y}_2, \qquad \wh{\Delta} \coloneq
  \frac{(\wh{\Pi}^{\T}  - \Pi^{*\T}) \y}{\nnorm{(\wh{\Pi}^{\T}  - \Pi^{*\T}) \y}_2} \notag \\
\Rightarrow \quad &\nnorm{(\wh{\Pi}^{\T}  - \Pi^{*\T}) \y}_2  \nnorm{\pr_X^{\perp}\wh{\Delta}}_2^2  \leq 2 |\nscp{\pr_X^{\perp}\eps}{\wh{\Delta}}| \label{eq:Pi_apply_Gordon}.  
\end{align}
\emph{i) Lower bounding the left hand side of \eqref{eq:Pi_apply_Gordon}.}\\
Observe that $\wh{\Delta} \in \mc{T}$ as defined in \eqref{eq:setT}. Moreover, note that under (\text{\textbf{G}}), the random subspace
$\text{range}(X) \subseteq \R^n$ follows the uniform distribution on the Grassmannian $\textsf{G}(n,d)$. Hence, applying Lemma \ref{lem:escape} with $K = \mc{T}$,
$V = \text{range}(X)$ and thus $m = n$ and $l = n-d$, we have for any $\epss \in (0,1/2)$ in compliance with \eqref{eq:cond_samples_unrelaxed}    
\begin{align}\label{eq:lowerbound}
  \begin{split}
   \p(\text{dist}(\mc{T}, \text{range}(X)) > \epss)
   &= \p \left(\min_{\Delta \in \mc{T}} \nnorm{\pr_X^{\perp} \Delta}_2 > \epss \right) \\
   &\geq 1 - \frac{7}{2} \exp \left(-\frac{1}{2} \left( \frac{(1 - \epss)\nu_{n-d} - \epss \nu_n  - w(\mc{T})}{3 + \epss + \epss \nu_n / \nu_{n-d}} \right)^2 \right) \\
   &\geq 1 - \frac{7}{2} \exp \left(-\frac{\psi(\epss, d/n) \nu_n^2}{8} \right), \quad
   \psi(\epss, d/n) \coloneq \left( \frac{(1- \epss) \frac{\nu_{n-d}}{\nu_n} - \epss}{3 + \epss + \frac{\epss^2}{1-\epss}} \right)^2,
 \end{split}
\end{align}
where the last inequality follows from condition \eqref{eq:cond_samples_unrelaxed}. Conditional on the event $\{ \nnorm{\pr_X^{\perp} \Delta}_2 \geq \epss \}$, which holds with the probability as stated, the left hand side of \eqref{eq:Pi_apply_Gordon} is lower bounded by $\epss^2 \nnorm{(\wh{\Pi}^{\T}  - \Pi^{*\T}) \y}_2$.  
\vskip1ex
\noindent \emph{ii) Upper bounding the right hand side of \eqref{eq:Pi_apply_Gordon}.}\\
Consider $t > 0$ arbitrary, to be chosen later.  We have
\begin{align*}
\p(|\nscp{\pr_X^{\perp}\eps}{\wh{\Delta}}| \geq t) \leq \p \left(\sup_{\Delta \in \mc{T}} |\nscp{\pr_X^{\perp}\eps}{\Delta}| > t \right). 
\end{align*}
\noindent Now note that $\eps$ and $\pr_X^{\perp}$ are independent. As a result, when conditioning on $\pr_X^{\perp}$, we have from Lemma \ref{lem:concentration_width}
that
\begin{align*}
  \p \left(\sup_{\Delta \in \mc{T}} |\nscp{\pr_X^{\perp}\eps}{\Delta}| > t \;\;\Big| \pr_X^{\perp} \right) &\leq
  \p \left(\sup_{\Delta \in \mc{T}} |\nscp{\pr_X^{\perp}\eps}{\Delta}| > \E\left[\sup_{\Delta \in \mc{T}}  |\nscp{\pr_X^{\perp}\eps}{\Delta}| \;\; \Big|  \pr_X^{\perp} \right] +  \tau \;\;\Big| \pr_X^{\perp} \right) \\
                                                                                                           &\leq \p \left(\sup_{\Delta \in \mc{T}} |\nscp{\pr_X^{\perp}\eps}{\Delta}| > \sigma  w(\mc{T})  +  \tau \;\;\Big| \pr_X^{\perp} \right) \\
                                                                                                           &\leq \exp(-\tau^2 / (2 \sigma^2)),
                                                                                                             \quad \tau \coloneq t - \E\left[\sup_{\Delta \in \mc{T}}  |\nscp{\pr_X^{\perp}\eps}{\Delta}| \;\; \Big|  \pr_X^{\perp} \right] \\
&\qquad \qquad \qquad \qquad \qquad \;\; \geq  t - \sigma w(\mc{T}). 
\end{align*}
The second equality is a consequence of the Sudakov-Fernique comparison inequality (Theorem 2.2.3 in \cite{AdlerTaylor2007}) and the fact that
$\nnorm{\pr_X^{\perp}}_2 \leq 1$. Choosing $t = (1 + \sqrt{2}) \sigma \{w(\mc{T}) \vee \log n \}$, combining this result with i), and putting together the pieces in \eqref{eq:Pi_apply_Gordon}, the assertion of the lemma follows.   
\end{proof}
\noindent In order to complete the proof of Theorem \ref{theo:constrained_permutation}, we need one last lemma whose proof is deferred to Appendix \ref{app:Xteps}.

\begin{lemma}\label{lem:Xteps} Let $\sigma_{\min}(\cdot)$ denote the minimum singular value of a matrix. If $n \geq 9 \vee 4d$, it holds that
\begin{align*}
  &\p \left( \left \{\sigma_{\min}(X / \sqrt{n}) < 1 - \sqrt{\frac{4 d \vee \log n}{n}} \right \}   \cup \left\{ \norm{ \left(\frac{X^{\T} X}{n} \right)^{-1} \frac{X^{\T} \eps}{\sqrt{n}}}_2 \geq \sigma \frac{\sqrt{5 (d \vee \log n)} \}}{1 - \sqrt{\frac{4 d \vee \log n}{n}}} \right \} \right)\\
  &\leq 2 \exp \left(-\frac{1}{2} (d \vee \log n) \right).  
\end{align*}
\end{lemma}
\noindent Equipped with Lemmas \ref{lem:bound_pi} and \ref{lem:Xteps}, we are in position to conclude
Theorem \ref{theo:constrained_permutation}. Expanding the expression for $\wh{\beta}$ in \eqref{eq:decomposition}, we obtain that
\begin{align*}
  \wh{\beta} = \left(\frac{X^{\T} X}{n} \right)^{-1} \frac{X^{\T} \wh{\Pi}^{\T} \y}{n} 
  &= \left(\frac{X^{\T} X}{n} \right)^{-1} \frac{X^{\T} \Pi^{*\T} \y}{n} + \left(\frac{X^{\T} X}{n} \right)^{-1} \frac{X^{\T} (\wh{\Pi}^{\T}  - \Pi^{*\T}) \y}{n} \\
  &= \beta^* + \left( \frac{X^{\T} X}{n} \right)^{-1} \frac{X^{\T} (\eps + (\wh{\Pi}^{\T} - \Pi^{*\T}) \y)}{n},  
\end{align*}
and thus
\begin{align*}
  \nnorm{\wh{\beta} - \beta^*}_2 &\leq  \frac{\norm{ \left(\frac{X^{\T} X}{n} \right)^{-1} \frac{X^{\T} \eps}{\sqrt{n}}}_2}{\sqrt{n}} + \frac{\norm{\left(\frac{X^{\T} X}{n} \right)^{-1} \frac{X^{\T}}{\sqrt{n}}}_2 \nnorm{(\wh{\Pi}^{\T} - \Pi^{*\T}) \y}_2}{\sqrt{n}} \\
                                 &\leq \frac{\norm{ \left(\frac{X^{\T} X}{n} \right)^{-1} \frac{X^{\T} \eps}{\sqrt{n}}}_2}{\sqrt{n}} + \frac{\frac{1}{\sigma_{\min}(X / \sqrt{n})}\nnorm{(\wh{\Pi}^{\T} - \Pi^{*\T}) \y}_2}{\sqrt{n}} \\
&\leq \frac{\sigma}{1 - \sqrt{\frac{4 d \vee \log n}{n}}} \left(\sqrt{\frac{5 (d \vee \log n)}{n}} + \frac{2 (1 + \sqrt{2})\sigma \epss^{-2} \{w(\mc{T}) \vee \log n\}}{\sqrt{n}} \right)                                   
\end{align*}
with the stated probability by combining Lemmas \ref{lem:bound_pi} and \ref{lem:Xteps}.

\section{Proof of Corollary \ref{corro:refitting}}

Observe that the event $\{\wh{S} = S^* \}$ is implied by the event
\begin{equation*}
\left\{ \min_{i \in S^*} |y_i - y_{\varphi^{-1}(i)}| > \nnorm{\Pi^{*\T} \y - \wh{\Pi}^{\T} \y}_{\infty} \right \},  
\end{equation*}
which is in turn implied by
\begin{equation*}
\left\{ \min_{i \in S^*} |y_i - y_{\varphi^{-1}(i)}| > \nnorm{\Pi^{*\T} \y - \wh{\Pi}^{\T} \y}_{2} \invcoloneq \gamma \right\},  
\end{equation*} 
We now have 
\begin{align*}
\p\left(\min_{i \in S^*} |y_i - y_{\varphi^{-1}(i)}| \leq \gamma \right) 
&\leq \sum_{i \in S^*} \p\left(|y_i - y_{\varphi^{-1}(i)}| \leq \gamma \right) \\
&\leq k  \max_{i \neq j} \p\left(|y_i - y_{j}| \leq \gamma \right).
\end{align*}
Note that for any $i \neq j$, $y_i - y_j \sim N(0, 2 (\nnorm{\beta^*}_2^2 + \sigma^2))$. Invoking 
Lemma \ref{lem:smallball}, we thus obtain that for any $i \neq j$ 
\begin{equation*}  
\p\left(|y_i - y_{j}| \leq \gamma \right) \leq \frac{\gamma}{\sqrt{2 (\nnorm{\beta^*}_2^2 + \sigma^2)}}.
\end{equation*}
It hence follows from the previous two displays that for any $\delta > 0$, 
\begin{equation*}
\p\left(\min_{i \in S^*} |y_i - y_{\varphi^{-1}(i)}| \leq \gamma \right) \leq \delta
\end{equation*}
if 
\begin{equation*}
\nnorm{\beta^*}_2^2 > \frac{k^2 \gamma^2}{2 \delta^2}. 
\end{equation*}
Substituting $\gamma$ by the bound from Lemma \ref{lem:bound_pi}, the above condition becomes
\begin{equation*}
\text{SNR} > \frac{2 (1 + \sqrt{2})^2 \epss^{-4}}{\delta^2}\,\cdot \frac{k^2  \{w(\mc{T}) \vee \log n\}^2}{n },
\end{equation*}    
and the first part of the corollary follows from Lemma \ref{lem:bound_pi}. Turning to the second part, consider
the least squares estimator using the reduced data set $\{(\x_i, y_i), \; i \in [n] \setminus S^* \}$ and denote
by $\wh{\beta}(S^*)$ the corresponding least squares estimator. We then have for any $\delta > 0$ 
\begin{equation*}
\p(\nnorm{\wh{\beta}(S^*) - \beta^*}_2 \leq \delta) = \p(\nnorm{\wh{\beta}(S^*) - \beta^*}_2 \leq \delta, \wh{S} = S^*) + 
\p(\nnorm{\wh{\beta}(S^*) - \beta^*}_2 \leq \delta, \wh{S} \neq S^*),
\end{equation*} 
which implies 
\begin{equation*}
\p(\nnorm{\wh{\wh{\beta}} - \beta^*}_2 \leq \delta) \geq \p(\nnorm{\wh{\beta}(S^*) - \beta^*}_2 \leq \delta) - \p(\wh{S} \neq S^*). 
\end{equation*}
It thus remains to control the first probability on the left hand side, which can be done with arguments similar 
to those used in the last part of the proof of Theorem \ref{theo:constrained_permutation}.

\section{Proof of Theorem \ref{theo:relaxation}}\label{app:theo:relaxation}

We start with a Lemma that parallels Lemma \ref{lem:decoupling}. The reasoning is analogous, and the proof is hence omitted.   
\begin{lemma}\label{lem:decoupling_LAD} 
Consider optimization problem \eqref{eq:relax_2} with solution $(\wt{\beta}, \wt{e})$. Then, if $n \geq d$, with probability one
\begin{align*}
  \wt{e} \in \argmin_{e} \frac{1}{n} \nnorm{\pre_X^{\perp} (\y - \sqrt{n}  e)}_2^2 + \lambda \nnorm{e}_1, \quad
  \wt{\beta} \in \Bigg\{ &\left(\frac{X^{\T} X}{n} \right)^{-1} \frac{X^{\T} (\y - \sqrt{n} \wt{e})}{n}, \\
  &\wt{e} \in \argmin_{e} \frac{1}{n} \nnorm{\pre_X^{\perp} (\y - \sqrt{n} e)}_2^2 + \lambda \nnorm{e}_1 \Big \}.
\end{align*}
\end{lemma}
\noindent For our analysis of the optimization problem of Lemma \ref{lem:decoupling_LAD}, we need the following
crucial lemma whose proof is provided in Appendix \ref{app:lem:RE}.  
\begin{lemma}\label{lem:RE} There exists constants $c_1, c_2, \epss > 0$ so that
  if $k \leq c_1 \frac{n-d}{\log(n/k)}$, the following event holds with probability at least $1 - 2 \exp(-c_2 (n-d) )$. 
\begin{equation*}
  \left\{ \frac{1}{n} \nnorm{\pre_X^{\perp} v}_2^2 \geq \epss \frac{n-d}{n} \nnorm{v}_2^2, \quad v \in \mc{C}(S^*, 3)
  \right \},
\end{equation*}
with $S^* = \{j:\;\Pi_{jj}^* \neq 1 \}$ and $\mc{C}(S^*, 3)$ defined in \eqref{eq:lassocone}. 
\end{lemma}
\noindent In correspondence with Lemma \ref{lem:bound_pi}, we first bound $\nnorm{\wt{e} - e^*}_2$,
where $e^* = (\Pi^* X \beta^* - X \beta^*)/\sqrt{n}$.
\begin{lemma}\label{lem:bound_e} Under the conditions of Theorem \ref{theo:relaxation}, with probability at least $1 - 2 \exp(-c_{\epss} (n-d))  - 2n^{-M^2}$, we have 
\begin{equation*}
\nnorm{\wt{e} - e^*}_2 \leq   48 (1 + M) \sigma  \frac{n}{n-d} \epss^{-1} \sqrt{\frac{k \log n}{n}}.  
\end{equation*}
\end{lemma}
\begin{proof} We first decompose
\begin{equation*}
  \pr_X^{\perp} \y = \pr_X^{\perp} (X \beta^* + \sqrt{n} e^* + \eps) = \pr_X^{\perp} \sqrt{n} e^*  + \wt{\eps}, \quad
  \wt{\eps} \coloneq \pr_X^{\perp} \eps.  
\end{equation*}
According to Lemma \ref{lem:decoupling_LAD}, since $\wt{e}$ is a minimizer, we have that
\begin{equation*}
  \frac{1}{n} \nnorm{\pr_X^{\perp} \sqrt{n}   (e^* - \wt{e}) + \wt{\eps}}_2^2 + \lambda \nnorm{\wt{e}}_1 \leq
  \frac{1}{n} \nnorm{\wt{\eps}}_2^2  + \lambda \nnorm{e^*}_1. 
\end{equation*}
Expanding the square on the left-hand side and re-arranging terms, we obtain that
\begin{equation}\label{eq:lasso_basic}
\frac{1}{n} \nnorm{\pr_X^{\perp} \sqrt{n}   (e^* - \wt{e})}_2^2 + \lambda \nnorm{\wt{e}}_1 \leq 2 |\scp{e^* - \wt{e}}{\wt{\eps}/\sqrt{n}}| + \lambda \nnorm{e^*}_1, 
\end{equation}
which implies
\begin{equation*}
 \lambda \nnorm{\wt{e}}_1 \leq 2 \nnorm{e^* - \wt{e}}_{1} \nnorm{\wt{\eps}/\sqrt{n}}_{\infty} + \lambda \nnorm{e^*}_1,
\end{equation*}
and in turn
\begin{align*}
  \lambda \nnorm{\wt{e}_{(S^*)^c} - e_{(S^*)^c}^*}_1 = \lambda \nnorm{\wt{e}_{(S^*)^c}}_1 &\leq 2 \nnorm{e^* - \wt{e}}_{1} \nnorm{\wt{\eps}/\sqrt{n}}_{\infty} + \lambda (\nnorm{e^*}_1 - \nnorm{\wt{e}_{S^*}}_1) \\
  &\leq  2 \nnorm{e^* - \wt{e}}_{1} \underbrace{\nnorm{\wt{\eps}/\sqrt{n}}_{\infty}}_{\invcoloneq \lambda_0} + \lambda \nnorm{e_{S^*}^*- \wt{e}_{S^*}}_1. 
\end{align*}
Re-arranging, we arrive at 
\begin{equation*}
(\lambda - 2\lambda_0) \nnorm{\wt{e}_{(S^*)^c} - e_{(S^*)^c}^*}_1 \leq (\lambda + 2 \lambda_0) \nnorm{e_{S^*}^*- \wt{e}_{S^*}}_1. 
\end{equation*}
Conditional on the event $\Lambda =\{\lambda \geq 4 \lambda_0 \}$, we hence have that
\begin{equation}\label{eq:cone_condition}
  \nnorm{\wt{e}_{(S^*)^c} - e_{(S^*)^c}^*}_1 \leq 3 \nnorm{\wt{e}_{S^*} - e_{S^*}^*}_1.
\end{equation}
\noindent Equipped with this is intermediate result, we go back to \eqref{eq:lasso_basic}. We first
obtain by re-arranging that 
\begin{align*}
  \frac{1}{n} \nnorm{\pr_X^{\perp} \sqrt{n}   (e^* - \wt{e})}_2^2 &\leq (2 \lambda_0 + \lambda) \nnorm{\wt{e} - e^*}_1 \\
  &\leq (2 \lambda_0 + \lambda) 4 \nnorm{\wt{e}_{S^*} - e_{S^*}^*}_1 \\
  &\leq (2 \lambda_0 + \lambda) 4 \sqrt{k} \nnorm{\wt{e}_{S^*} - e_{S^*}^*}_2, 
\end{align*}
where we have used \eqref{eq:cone_condition} in the second inequality.  
In order to lower bound the left hand side, note that \eqref{eq:cone_condition} can be written as
$\wt{e} - e^* \in \mc{C}(S^*, 3)$ according to \eqref{eq:lassocone}. Conditional on the event in Lemma
\ref{lem:RE},
\begin{equation*}
\frac{1}{n} \nnorm{\pr_X^{\perp} \sqrt{n}   (e^* - \wt{e})}_2^2 \geq  \epss \frac{n-d}{n}  \nnorm{e^* - \wt{e}}_2^2.
\end{equation*}
Combining this with the previous display, we find that
\begin{equation*}
 \nnorm{e^* - \wt{e}}_2 \leq 12  \epss^{-1} \frac{n}{n-d} \lambda \sqrt{k}. 
\end{equation*}
\noindent In view of Lemma \ref{lem:maximalgaussian}, the choice $\lambda = 4 (1 + M) \sigma \sqrt{2 \log(n) / n}$ for $M \geq 0$
guarantees that the event $\Lambda$ defined above \eqref{eq:cone_condition} occurs with probability at least $1 - 2n^{-M^2}$.
This completes the proof of the Lemma. 
\end{proof}
\noindent Using Lemmas \ref{lem:bound_pi} and \ref{lem:decoupling_LAD}, the proof of the theorem is completed along the lines
of the last part of the proof of Theorem \ref{theo:constrained_permutation}. For brevity, we omit those steps here.

\section{Proof of Theorem \ref{theo:permutationrecovery}}
The proof of part (a) of the theorem relies on two lemmas, which are stated and proved first. 
\begin{lemma}\label{lem:up}
Let $\gamma, \delta > 0$. Consider the event 
\begin{equation}\label{eq:Agamma}
\mathcal{A}_{\gamma}= \left\{\min_{i<j}(\x_i^\top\beta^*-\x_j^\top\beta^*)^2\le \gamma^2 \right\}.
\end{equation}
Then $\p\left( \mathcal{A}_{\gamma}\right)< \delta$ if 
\[
\nnorm{\beta^*}_2^2>\dfrac{n^2(n-1)^2}{4\delta^2\pi}\gamma^2.
\]	
\end{lemma}

\begin{proof}
It is clear that $\mathcal{A}_{\gamma}\subseteq\bigcup_{i<j}\left\{(\x_i^\top\beta^*-\x_j^\top\beta^*)^2\le \gamma^2 \right\}$. Using the fact that $\x_i^\top\beta^*-\x_j^\top\beta^*\sim N(0,2\|\beta^*\|_2^2)$, we have
\begin{align*}
	\p\left(\mathcal{A}_{\gamma}\right)&\le\p\left(\bigcup_{i<j}\left\{-\dfrac{\gamma}{\sqrt{2}\|\beta^*\|_2}\le (\x_i -\x_j)^\top\dfrac{\beta^*}{\sqrt{2}\|\beta^*\|_2}\le \dfrac{\gamma}{\sqrt{2}\|\beta^*\|_2}\right\}\right)	\\
	&\le \sum_{i<j}\p\left( -\dfrac{\gamma}{\sqrt{2}\|\beta^*\|_2} \le g \leq \dfrac{\gamma}{\sqrt{2}\|\beta^*\|_2} \right), \quad g \sim N(0,1)\\
	&\le\dfrac{n(n-1)}{2} \dfrac{\gamma}{\|\beta^*\|_2\sqrt{\pi}} \quad \text{by Lemma \ref{lem:smallball} in Appendix \ref{app:aux}.}
\end{align*}
Thus, it suffices to require that $\dfrac{n(n-1)}{2} \dfrac{\gamma}{\|\beta^*\|_2\sqrt{\pi}}<\delta$. Resolving for $\nnorm{\beta^*}_2$, this condition becomes
\[
\|\beta^*\|_2^2>\dfrac{n^2(n-1)^2}{4\delta^2\pi}\gamma^2.
\]
\end{proof}
\noindent The following lemma is of interest in its own when the design matrix is fixed and both $\beta^*$ and $\sigma^2$ are known or accurate estimates of these quantities are available. Condition \eqref{eq:fix} can then be evaluated explicitly, at least after substituting $\beta^*$ and $\sigma^2$ by their respective estimates.  The result is close in spirit to those in the framework in \cite{Collier2016}.
\begin{lemma}\label{lem:fixed}
Conditional on $X$, we have $\p \left(\widehat{\Pi}(\beta^*) \ne \Pi^*\mid X \right)<\delta$ if
\begin{equation}\label{eq:fix}
\min_{i<j}\left(\x_i^\top\beta^*-\x_j^\top\beta^*\right)^2>4\sigma^2\log\dfrac{n(n-1)}{\delta}.
\end{equation}
\end{lemma}

\begin{proof} Without loss of generality, we may assume that $\Pi^* = I_n$. By construction of $\widehat{\Pi}(\beta^*)$, it is clear
  that $\widehat{\Pi}(\beta^*) \ne \Pi^*$ whenever there exists $i\ne j$ such that $(y_i-y_j)(\x_i^\top\beta^*-\x_j^\top\beta^*)<0$, that is, the order between $\x_i^\top\beta^*$ and $\x_j^\top\beta^*$ is different than that of $y_i, y_j$. First, conditioning on the random design matrix $X$,  we have
\begin{align*}
\p(\wh{\Pi}(\beta^*)\ne \Pi^* \mid X)&=2\p\left(\bigcup_{i<j}\left\{y_i-y_j>0 \mid \x_i^\top\beta^*-\x_j^\top\beta^*<0\right\}\right)\\\
&\le 2 \sum_{i<j}\p\left(y_i-y_j>0 \mid \x_i^\top\beta^*-\x_j^\top\beta^*<0\right)
\end{align*}
 Since $y_i-y_j\mid \x_i^\top\beta^*-\x_j^\top\beta^*\sim N(\x_i^\top\beta^*-\x_j^\top\beta^*, 2\sigma^2)$, using the usual tail bound for the Gaussian distribution for each term in the above sum, we obtain
 \begin{align*}
\p(\wh{\Pi}(\beta^*)\ne \Pi^* \mid X)&\le 2\sum_{i<j} \exp\left( -\dfrac{\left(\x_i^\top\beta^*-\x_j^\top\beta^*\right)^2}{4\sigma^2} \right)\\
&\le n(n-1)\exp\left( -\dfrac{\min_{i<j}\left(\x_i^\top\beta^*-\x_j^\top\beta^*\right)^2}{4\sigma^2} \right).
\end{align*}
Hence, for any $\delta>0$, $\p(\wh{\Pi}(\beta^*)\ne \Pi^* \mid X) <\delta$ if

\begin{equation*}
\min_{i<j}\left(\x_i^\top\beta^*-\x_j^\top\beta^*\right)^2>4\sigma^2\log\dfrac{n(n-1)}{\delta}.
\end{equation*}
\end{proof}

\noindent Combining Lemmas \ref{lem:up} and \ref{lem:fixed}, part (a) in Theorem \ref{theo:permutationrecovery} follows
immediately for $\wh{\Pi}(\beta^*)$ from the previous lemma by invoking Lemma \ref{lem:up} with $\gamma = 2 \sigma \sqrt{\log\frac{n(n-1)}{\delta}}$. We then have
$\p(\wh{\Pi}(\beta^*)\ne \Pi^*)\le \p(\mathcal{A}_{\gamma})+\left(1-\p(\mathcal{A}_{\gamma})\right)\delta<2\delta$ under the conditions of the theorem.
The general version of the theorem results from
\begin{align*}
  \min_{i < j} |\x_i^{\T} \wh{\theta} - \x_j^{\T} \wh{\theta}| &\geq  \min_{i < j} |\x_i^{\T} \beta^* - \x_j^{\T} \beta^*| - 2 \max_{1 \leq i \leq n} |\x_i^{\T} (\beta^* - \wh{\theta})| \\
  &\geq \min_{i < j} |\x_i^{\T} \beta^* - \x_j^{\T} \beta^*| - 2 \sigma \Delta. 
\end{align*}
conditional on the event $\{ \nnorm{X \beta^* - X \wh{\theta}}_{\infty} \leq \sigma \Delta \}$. The assertion follows from an application of
Lemma \ref{lem:up} with $\gamma = 2 \sigma \left(\sqrt{\log\frac{n(n-1)}{\delta}} + \Delta \right)$. 
\vskip2ex
\noindent The proof of part (b) in Theorem \ref{theo:permutationrecovery} requires another lemma.

\begin{lemma}\label{lem:low}
Consider the event $\mc{A}_{\gamma}$ in \eqref{eq:Agamma}. Then, for all $n \geq 5$, $\p\left( \mathcal{A}_{\gamma}\right)\ge 0.85$ if 
\begin{equation}\label{eq:lb}
\|\beta^*\|_2^2\le \dfrac{1}{25\pi}n^2\gamma^2.
\end{equation}
\end{lemma}
\begin{proof}
Since $\x_i^\top\beta^*/\|\beta^*\|_2 \sim N(0,1)$, we have for $g \sim N(0,1)$
\begin{equation}\label{eq:ap1}
p:=\p\left( -\dfrac{\gamma}{\sqrt{2}}<\x^\top\beta^*<\dfrac{\gamma}{\sqrt{2}}\right)=\p\left( -\dfrac{\gamma}{\sqrt{2}\nnorm{\beta^*}_2}<g<\dfrac{\gamma}{\sqrt{2}\nnorm{\beta^*}_2} \right)\ge \dfrac{\gamma}{\nnorm{\beta^*}_2\sqrt{2\pi}},
\end{equation}
whenever $\gamma < \frac{3}{\sqrt{2}} \nnorm{\beta^*}_2$, otherwise $p > .86$ as can be verified by evaluating the Gaussian cumulative distribution function numerically. 

\noindent Let  $\mathcal{B}:=\left\{-\dfrac{\gamma}{\sqrt{2}}<\x^\top\beta^*<\dfrac{\gamma}{\sqrt{2}}\right\}$. For \textit{i.i.d.} observations $\x_1,\ldots,\x_n$, let
\[
u:=\sum_{i=1}^n \mathbb{I}(\x_i\in \mathcal{B})\sim \text{binom}(n, p),
\]
with $p$ as defined in \eqref{eq:ap1}. The event $\mathcal{A}_{\gamma}$ in \eqref{eq:Agamma} occurs whenever at least two $\x_i, \x_j\in \mathcal{B}$, i.e., the event  $\left\{u\ge 2\right\}\subseteq\mathcal{A}_{\gamma}$. Let us suppose for the moment that $p \geq 3.5/n$ for $n\ge 5$.  Then
\begin{align*}
	\p\left(\mathcal{A}_{\gamma}\right)&\ge\p\left(u\ge 2\right)\\
&\geq 1-\left(1-\dfrac{3.5}{n}\right)^{n}-n\left(\dfrac{3.5}{n}\right)\left(1-\dfrac{3.5}{n}\right)^{n-1}\\
&>.85 \quad \text{for $n\ge 5$}.
\end{align*}
Using Eq.\eqref{eq:ap1}, we now conclude that $\p\left(\mathcal{A}_{\gamma}\right)\ge .85$ whenever $\dfrac{\gamma}{\|\beta^*\|_2 \sqrt{2\pi}}\ge 3.5/n$. This holds if 
\[
\|\beta^*\|_2^2 \le \dfrac{1}{25\pi}n^2 \gamma^2.
\]
\end{proof}
\noindent We finally turn to the proof of part (b) of the theorem. Define indices $i_0, j_0 \subset [n]$ by the relation $\x_{i_0}^\top\beta^*-\x_{j_0}^\top\beta^*=\min_{i<j}|\x_i^\top\beta^*-\x_j^\top\beta^*|$. For given $\gamma>0$, we have
\begin{align*}
\p(\widehat{\Pi}(\beta^*)\ne \Pi^*) &\ge \p\left(y_{i_0}-y_{j_0} <0\mid 0<\x_{i_0}^\top\beta^*-\x_{j_0}^\top\beta^*<\gamma\right)\p\left(0<\x_{i_0}^\top\beta^*-\x_{j_0}^\top\beta^*<\gamma\right)\\
                                    &= \p\left(\x_{i_0}^\top\beta^*+\sigma\epsilon_{i_0}-\x_{j_0}^\top\beta^*-\sigma\epsilon_{j_{0}} <0\mid 0<\x_{i_0}^\top\beta^*-\x_{j_0}^\top\beta^*<\gamma\right) \times \\
  &\quad \times \p\left(0<\x_{i_0}^\top\beta^*-\x_{j_0}^\top\beta^*<\gamma\right)\\
&\ge \p\left(\sigma\epsilon_{i_0}-\sigma\epsilon_{j_{0}} <-\gamma \right)\p\left(0<\x_{i_0}^\top\beta^*-\x_{j_0}^\top\beta^*<\gamma\right)\\
&\ge \p\left(g<-\dfrac{\gamma}{\sqrt{2}\sigma} \right)\p\left(0<\x_{i_0}^\top\beta^*-\x_{j_0}^\top\beta^*<\gamma\right)\\
                                    &\ge  \p\left(g<-\dfrac{4\sqrt{\pi}\nnorm{\beta^*}_2}{\sqrt{2}\sigma n}c \right)\p\left(0<\x_{i_0}^\top\beta^*-\x_{j_0}^\top\beta^*<\dfrac{4\sqrt{\pi}\nnorm{\beta^*}_2}{n}c\right), \\
  &\quad\;\; \text{by choosing $\gamma=\dfrac{4\sqrt{\pi}\nnorm{\beta^*}_2}{n}c$, for $0<c\le 1$.}\\
&\ge \Phi(-.1)(.85), \: \text{by Lemma \ref{lem:low}, for a choice of $c$ s.t.$\dfrac{4\sqrt{\pi}\nnorm{\beta^*}_2}{\sqrt{2}\sigma n}c\le .1$.}\\
&\ge.35,
\end{align*}
where in the penultimate line, $\Phi(x) = \int_{-\infty}^x \frac{1}{\sqrt{2\pi}}\exp(-t^2/2) \; dt$ denotes the Gaussian cumulative distribution function.  
Therefore, $\p(\wh{\Pi}(\beta^*)\ne\Pi^*)\ge .35$ whenever $\frac{\nnorm{\beta^*}_2^2}{\sigma^2}<\frac{C}{n^2}$ for some $C>0$. This concludes
the proof of part (b) of the theorem. 


\section{Proof of Lemma \ref{lem:RE}}\label{app:lem:RE}
The proof of Lemma \ref{lem:RE} is given at the end of this section. The proof depends on several
lemmas collected from the literature which are stated first.     
\vskip1ex
\noindent For $1 \leq s \leq p$, let $\mc{J}(s)$ denote the set of all index subsets of $[p]$ of
size $s$.    

\begin{defn} For  $J \in \mc{J}(s)$, and $\alpha \in [1,\infty)$, define 
\begin{equation}\label{eq:lassocone}
\mc{C}(J, \alpha) = \{v \in \R^p:\, \nnorm{v_{J^c}}_1 \leq \alpha \nnorm{v_J}_1 \}. 
\end{equation}
\end{defn}

\begin{defn}\cite{CandesTao2005} Let $A \in \R^{N \times p}$ and $B_0(s;p) = \{v \in \R^p: \; \nnorm{v}_0 \leq s \}$. Define the $s$-sparse minimum and maximum eigenvalues associated with $A$ by 
\begin{equation*}
  \lambda_{\min}(s) = \min_{v \in B_0(s;p) \setminus \{0\}} \frac{\nnorm{A v}_2^2}{\nnorm{v}_2^2}, \qquad
  \lambda_{\max}(s) = \max_{v \in B_0(s;p) \setminus \{0\}} \frac{\nnorm{A v}_2^2}{\nnorm{v}_2^2}. 
\end{equation*}
Let $\delta \in (0,1)$. We say that $A$ is an ($s$, $\delta$)-\textbf{restricted isometry}  if
\begin{equation*}
  1 - \delta \leq \lambda_{\min}(s) \leq \lambda_{\max}(s) \leq 1+\delta.
\end{equation*}
\end{defn}

\begin{defn}\cite{Bickel2009} Let $A \in \R^{N \times p}$. We define the $(s, \alpha)$-\textbf{restricted eigenvalue} by 
\begin{equation}\label{eq:recondition}
\phi(s, \alpha) = \min_{J \in \mc{J}(s)} \, \min_{v \in \mc{C}(J, \alpha) \setminus
  0}\, \, \frac{\nnorm{A v}_2^2}{\nnorm{v_J}_2^2}. 
\end{equation}   
\end{defn}


\begin{lemma}\cite{Bickel2009} Suppose that $A$ is a ($3s, \delta^*/4$)-restricted isometry for some $\delta^* \in (0,1)$. The $(s, 3)$-restricted eigenvalue then satisfies $\phi(s, 3) \geq 1 - \delta^*$. 
\end{lemma}

\begin{lemma}\cite{Baraniuk2006}
Let $A$ be a random of matrix of size $N \times p$ satisfying 
\begin{equation}\label{eq:normpreservation}
 \p((1 - \epss) \nnorm{v}_2^2 \leq \nnorm{A v}_2^2 \leq (1 + \epss) \nnorm{v}_2^2) \leq 2 \exp(-N c_0(\epss))
\end{equation}
for any $\epss \in (0,1)$ and any $v \in \R^p$, where $c_0(\epss) > 0$ is a constant depending only on $\epss$. Then for any $\delta \in (0,1)$, there exist constants $c_1, c_2 > 0$ depending only on $\delta$ such that $A$ is a $(q, \delta)$-restricted isometry for any $q \leq c_1 N / \log(p/q)$ with probability at least $1 - 2 \exp(-c_2 N)$. 
\end{lemma}

\begin{lemma}\label{lem:DasGupta}(from Lemma 2.2 in \cite{DasGupta2003})
  Let $\pre$ denote the orthogonal projection on an $N$-dimensional subspace of $\R^p$ chosen uniformly at
  random from the Grassmannian $\textsf{G}(p, N)$. Then, for any $v \in \R^p$ and any $\epss \in (0,1)$,
  \begin{equation*}
    \p \left((1 - \epss) \nnorm{v}_2^2 \leq \frac{p}{N} \norm{\pre v}_2^2 \leq (1 + \epss) \nnorm{v}_2^2 \right) \leq 2 \exp(-N  (\epss^2/4 - \epss^3/6) \})
  \end{equation*}
\end{lemma}

\noindent\emph{\textbf{Proof of Lemma \ref{lem:RE}}}: the lemma is essentially implied by the fact that under the stated conditions, the matrix $\pr_{X}^{\perp}$ has its $(k,3)$-restricted eigenvalue bounded away from zero by a constant $\epss$ as to be shown below. "Essentially" refers to the fact that it can be shown (cf.~\cite{Zhou2009}) that a matrix $A$ with restricted eigenvalue $\phi(k,3)$ obeys
\begin{equation*}
\nnorm{A v}_2^2 \geq \frac{1}{32} \phi(k,3) \nnorm{v}_2^2 \qquad \forall v \in \bigcup_{J \in \mc{J}(k)} \mc{C}(J,3). 
\end{equation*}  
As sketched above, a lower bound on $\phi(k,3)$ follows if $A$ is a $(3k, \delta^*/4)$ restricted isometry for some $\delta^* \in (0,1)$. That property is in turn satisfied with high probability if $A$ satisfies \eqref{eq:normpreservation}. Applying Lemma \ref{lem:DasGupta} with 
$\texttt{P} = \pr_{X}^{\perp}$, $p = n$ and $N = n-d$ yields the conclusion.    

\section{Proof of Lemma \ref{lem:Xteps}}\label{app:Xteps}
We invoke Lemma \ref{lem:hansonwright} conditional on $X$, with $\left(\frac{X^{\T} X}{n} \right)^{-1} \frac{X^{\T}}{\sqrt{n}}$ in the role of $A$, so that
$\Gamma = \frac{X}{\sqrt{n}} \left(\frac{X^{\T} X}{n} \right)^{-2} \frac{X^{\T}}{\sqrt{n}}$. Observe that 
\begin{align}\label{eq:traceineq}
\begin{split}
&(i)\;\;\tr(\Gamma) = \tr\left( \left(\frac{X^{\T} X}{n} \right)^{-1} \right) \leq  \frac{d}{\sigma^2_{\min}(X / \sqrt{n})}, \quad (ii) \;\; \nnorm{\Gamma}_2 = \frac{1}{\sigma^2_{\min}(X / \sqrt{n})}, \\
&(iii)\;\;\sqrt{\tr(\Gamma^2)} = \sqrt{\tr \left(\left(\frac{X^{\T} X}{n} \right)^{-2} \right)} \leq  \frac{\sqrt{d}}{\sigma^2_{\min}(X/\sqrt{n})}.
\end{split}
\end{align}
In view of the relations in \eqref{eq:traceineq}, application of Lemma \ref{lem:hansonwright} with $t = d \vee \log n$ and comparison of terms yields that
\begin{equation*}
  \p\left(\norm{ \left(\frac{X^{\T} X}{n} \right)^{-1} \frac{X^{\T} \eps}{\sqrt{n}}}_2 \; >  \sigma \frac{\sqrt{5 \{d \vee \log n \}}}{\sigma_{\min}(X / \sqrt{n})} \; \; \Bigg| X \right)
  \leq \exp(-d \vee \log n). 
\end{equation*}
Using Lemma \ref{lem:extremesingularvalues} for $\sigma_{\min}(X)$ with the choice $t = \sqrt{d \vee \log n}$, the assertion readily follows.

\section{Additional lemmas}\label{app:aux}
We here collect various auxiliary results acting as lemmas in the proofs of our main results.

\begin{lemma}\label{lem:escape}\emph{(\textbf{Gordon's Escape through a Mesh theorem \cite{Gordon1988}})}
 Let $K$  be a closed subset of the unit sphere in $\R^m$, let $\nu_r = \E_{g \sim N(0, I_r)}[\nnorm{g}_2]$, and let 
  $\epss > 0$. If the Gaussian width \eqref{eq:gaussianwidth} of $K$ obeys $w(K) < (1 - \epss) \nu_l - \epss \nu_m$, then an $(m - l)$-dimensional subspace
  $V$ drawn uniformly from the Grassmannian $\textsf{G}(m, m-l)$ satisfies
  \begin{equation*}
\p(\text{\emph{dist}}(K, V) > \epss) \geq 1 - \frac{7}{2} \exp \left(-\frac{1}{2} \left( \frac{(1 - \epss)\nu_l - \epss \nu_m  - w(K)}{3 + \epss + \epss \nu_m / \nu_l} \right)^2 \right).
\end{equation*}
\end{lemma}

\begin{lemma}\label{lem:concentration_width}\emph{(\textbf{Concentration of Gaussian processes \cite{Boucheron2013}})}
Let $K$ be a closed subset of the unit sphere in $\R^m$ with Gaussian width $w(K)$, and let $g \sim N(0, I_m)$. Then for all $t > 0$
 \begin{equation*}
 \p \left(\sup_{x \in K} |\scp{g}{x}| \geq w(K) + t \right) \leq \exp(-t^2/2). 
\end{equation*}
\end{lemma}

\begin{lemma}\label{lem:normconcentration} \emph{(\textbf{Concentration of the norm of Gaussian random vectors \cite{Vershynin2010}})}\\ 
Let $g \sim N(0, I_n)$. Then for all $t > 0$, $\p(\nnorm{g}_2 \geq (1 + t) \sqrt{n}) \leq \exp(-t^2 n / 2)$.  
\end{lemma}

\begin{lemma}\label{lem:extremesingularvalues}\hspace*{-.7ex}\emph{(\textbf{Concentration of extreme singular values of Gaussian matrices \cite{Vershynin2010}})}\\
  Let $X$ be an $n \times d$ Gaussian matrix with i.i.d.~$N(0,1)$-entries. Denote by $\sigma_{\min}(X)$
  and $\sigma_{\max}(X)$ the minimum respectively maximum singular value of $X$. Then for any $t > 0$
\begin{equation*}
\p(\sigma_{\min}(X) \geq \sqrt{n} - \sqrt{d} - t) \leq \exp(-t^2 / 2), \qquad \p(\sigma_{\max}(X) \geq \sqrt{n} + \sqrt{d} + t) \leq \exp(-t^2 / 2).  
\end{equation*}
\end{lemma}

\begin{lemma}\label{lem:hansonwright}\emph{(\textbf{Concentration of quadratic forms \cite{Hsu2012}})}\\
Let $A$ be a an $m \times n$ matrix, $\Gamma = A^{\T} A$, and $g \sim N(0, I_n)$. Then for all $t > 0$
\begin{equation*}
  \p(\nnorm{A g}_2^2  > \tr(\Gamma) + 2 \sqrt{\tr(\Gamma^2) t} + 2 \nnorm{\Gamma}_2 t) \leq \exp(-t). 
\end{equation*}
\end{lemma}

\begin{lemma}\label{lem:maximalgaussian}\emph{(\textbf{Maximal inequality for Gaussian random variables})}\\
Let $g_1, \ldots, g_n$ be zero-mean Gaussian random variables such that $\max_{1 \leq i \leq n} \E[g_i^2] \leq \sigma^2$. Then
for all $M > 0$, $\p(\max_{1 \leq i \leq n} |g_i| > (1+M) \sigma \sqrt{2 \log n}) \leq 2n^{-M^2}$. 
\end{lemma}

\begin{lemma}\label{lem:smallball} \emph{(\textbf{Small ball probability})}\\
Let $g \sim N(0,1)$. Then for all $t > 0$, $\p(|g| < t) \leq \sqrt{\frac{2}{\pi}} \, t < t$.  
\end{lemma}

\end{document}